\documentclass{article}
\usepackage[utf8]{inputenc}
\usepackage[T1]{fontenc}
\usepackage{amsmath,amssymb,amsthm}
\usepackage[shortlabels]{enumitem}
\usepackage{color,url}
\usepackage{xcolor}
\usepackage{hyperref}
\usepackage[capitalise]{cleveref}
\usepackage{xspace}
\usepackage{todonotes}
\usepackage{thm-restate}

\usepackage{colortbl}
\usepackage{tabulary}
\usepackage{etoolbox}
\usepackage{authblk}

\theoremstyle{plain}
\newtheorem{theorem}{Theorem}
\newtheorem{lemma}{Lemma}
\newtheorem{observation}{Observation}

\newtheorem{claim}{Claim}[theorem]
\newenvironment{claimproof}{\noindent\textit{Proof of Claim \theclaim:}}{\hfill$\lrcorner$} % To remove numbering, delete '\theclaim'

\renewcommand{\epsilon}{\varepsilon}

\DeclareMathOperator{\dist}{dist}
\DeclareMathOperator{\diam}{diam}
\DeclareMathOperator{\BCSP}{BCSP}

\newcommand{\Oh}{\mathcal{O}}
\newcommand{\NP}{$\mathsf{NP}$}
\newcommand{\p}{\mathsf{p}}

\newcommand{\homo}[1]{\textsc{Hom}(\ensuremath{#1})\xspace}
\newcommand{\lhomo}[1]{\textsc{LHom}(\ensuremath{#1})\xspace}

\newcommand{\cR}{\mathcal{R}}

\newcommand{\vphi}{\varphi}

\usepackage{todonotes}

\title{$C_{2k+1}$-coloring of bounded-diameter graphs}
\author{Marta Piecyk\thanks{This research was funded by Polish National Science Centre, grant no. 2022/45/N/ST6/00237.} \\ \texttt{marta.piecyk.dokt@pw.edu.pl}}
\affil{Faculty of Mathematics and Information Science, Warsaw University of Technology}
\date{}
\begin{document}
\maketitle

\begin{abstract}
For a fixed graph $H$, in the graph homomorphism problem, denoted by $\homo{H}$, we are given a graph $G$ and we have to determine whether there exists an edge-preserving mapping $\vphi: V(G) \to V(H)$. 
Note that $\homo{C_3}$, where $C_3$ is the cycle of length $3$, is equivalent to $3$-\textsc{Coloring}.
The question of whether $3$-\textsc{Coloring} is polynomial-time solvable on diameter-$2$ graphs is a well-known open problem.
In this paper we study the $\homo{C_{2k+1}}$ problem  on bounded-diameter graphs for $k\geq 2$, so we consider all other odd cycles than $C_3$.
We prove that for $k\geq 2$, the $\homo{C_{2k+1}}$ problem is polynomial-time solvable on diameter-$(k+1)$ graphs -- note that such a result for $k=1$ would be precisely a polynomial-time algorithm for $3$-\textsc{Coloring} of diameter-$2$ graphs.

Furthermore, we give subexponential-time algorithms for diameter-$(k+2)$ graphs.

We complement these results with a lower bound for diameter-$(2k+2)$ graphs -- in this class of graphs the $\homo{C_{2k+1}}$ problem is NP-hard and cannot be solved in subexponential-time, unless the ETH fails.

Finally, we consider another direction of generalizing $3$-\textsc{Coloring} on diameter-$2$ graphs. We consider other target graphs $H$ than odd cycles but we restrict ourselves to diameter $2$. We show that if $H$ is triangle-free, then $\homo{H}$ is polynomial-time solvable on diameter-$2$ graphs.

%\keywords{graph homomorphism, diameter, odd cycles}
\end{abstract}

%\marta[inline]{TO DO: 1. algorytm podwykładniczy k+2 i k+3, 2. przejść przez sekcję z obserwacjami, 4. sprawdzić dolne ograniczenie i zrobić obrazki, 5. dopisać sekcję o grafach bez trójkątów, 6. abstrakt, 7. wstęp}

\section{Introduction}
A natural approach to computationally hard problems is to restrict the class of input graphs, for example by bounding some parameters. One of such parameters is the diameter, i.e., for a graph $G$, its diameter is the least integer $d$ such that for every pair of vertices $u,v$ of $G$, there is a $u$-$v$ path on at most $d$ edges. We say that $G$ is a diameter-$d$ graph, if its diameter is at most $d$. Recently, bounded-diameter graphs received a lot of attention~\cite{DBLP:conf/mfcs/MartinPS19,DBLP:conf/wads/KlimosovaS23,DBLP:journals/algorithmica/MertziosS16,DBLP:journals/combinatorics/BrauseGMOPS22,DBLP:journals/ipl/BonamyDFJP18,DBLP:journals/dam/MartinPS22,DBLP:journals/siamdm/DebskiPR22,DBLP:journals/combinatorics/CamposGILSS21}. It is known that graphs from real life applications often have bounded diameter, for instance social networks tend to have bounded diameter~\cite{DBLP:journals/socnet/Schnettler09}. Furthermore, almost all graphs have diameter $2$, i.e., the probability that a random graph on $n$ vertices has diameter $2$ tends to $1$ when $n$ tends to infinity~\cite{korshunov}. Therefore, solving a problem on bounded-diameter graphs captures a wide class of graphs. On the other hand, not all of the standard approaches can be used on bounded-diameter -- note that the class of diameter-$d$ graphs is not closed under vertex deletion.

Even if we consider the class of diameter-$2$ graphs, its members can contain any graph as an induced subgraph. Indeed, consider any graph $G$, and let $G^+$ be the graph obtained from $G$ by adding a universal vertex $u$, i.e., we add vertex $u$ and make it adjacent to all vertices of $G$. It is straightforward to observe that the diameter of $G^+$ is at most $2$. This construction can be used for many graph problems on diameter-$2$ graphs as a hardess reduction, which proves that they cannot be solved in subexponential time under the Exponential Time Hypothesis (ETH, see~\cite{DBLP:journals/jcss/ImpagliazzoP01}), for instance \textsc{Max Independent Set}.

The construction of $G^+$ also gives us a reduction from $(k-1)$-\textsc{Coloring} to $k$-\textsc{Coloring} on diameter-$2$ graphs, and thus for any $k\geq 4$, the $k$-\textsc{Coloring} problem on diameter-$2$ graphs is \NP-hard and cannot be solved in subexponential time, unless the ETH fails. Note that this argument does not work for $k=3$ since then we reduce from $2$-\textsc{Coloring}, which is polynomial-time solvable. A textbook reduction from \textsc{NAE-Sat} implies that for $d\geq 4$, the $3$-\textsc{Coloring} problem is \NP-hard and cannot be solved in subexponential-time, unless the ETH fails~\cite{DBLP:books/daglib/0072413}. Therefore, it is only interesting to study $3$-\textsc{Coloring} on diameter-$2$ and -$3$ graphs. Mertzios and Spirakis proved that $3$-\textsc{Coloring} is \NP-hard on diameter-$3$ graphs~\cite{DBLP:journals/algorithmica/MertziosS16}. However, the question of whether $3$-\textsc{Coloring} can be solved in polynomial time on diameter-$2$ graphs remains open.

For $3$-\textsc{Coloring} on diameter-$2$ graphs, there were given subexponential-time algorithms, first by Mertzios and Spirakis with running time $2^{\Oh(\sqrt{n\log{n}})}$~\cite{DBLP:journals/algorithmica/MertziosS16}. This was later improved by Dębski, Piecyk, and Rzążewski, who gave an algorithm with running time $2^{\Oh(n^{1/3}\cdot \log^2n)}$~\cite{DBLP:journals/siamdm/DebskiPR22}. They also provided a subexponential-time algorithm for $3$-\textsc{Coloring} for diameter-$3$ graphs.

The $3$-\textsc{Coloring} problem on bounded-diameter graphs was also intensively studied on instances with some additional restrictions, i.e., on graphs with some forbidden induced subgraphs -- and on such graph classes polynomial-time algorithms were given~\cite{DBLP:conf/mfcs/MartinPS19,DBLP:conf/wads/KlimosovaS23,DBLP:journals/dam/MartinPS22}.

One of the generalizations of graph coloring are homomorphisms. For a fixed graph $H$, in the graph homomorphism problem, denoted by $\homo{H}$, we are given a graph $G$, and we have to determine whether there exists an edge-preserving mapping $\vphi: V(G)\to V(H)$, i.e., for every $uv\in E(G)$, it holds that $\vphi(u)\vphi(v)\in E(H)$. Observe that for $K_k$ being a complete graph on $k$ vertices, the $\homo{K_k}$ problem is equivalent to $k$-\textsc{Coloring}.
Observe also that the problem is trivial when $H$ contains a vertex $x$ with a loop since we can map all vertices of $G$ to $x$. In case when $H$ is bipartite, in fact we have to verify whether $G$ is bipartite and this can be done in polynomial time. Hell and Ne\v{s}et\v{r}il proved that for all other graphs $H$, i.e., loopless and non-bipartite, the $\homo{H}$ problem is \NP-hard~\cite{HELL199092}.
Such a complete dichotomy was provided by Feder, Hell, and Huang also for the list version of the problem~\cite{DBLP:journals/jgt/FederHH03}.
The graph homomorphism problem and its variants in various graph classes and under various parameterizations received recently a lot of attention~\cite{DBLP:conf/esa/OkrasaPR20,DBLP:conf/icalp/GanianHKOS22,DBLP:conf/soda/FockeMR22,DBLP:conf/stoc/BulatovK22,DBLP:conf/stoc/CaiM23,DBLP:journals/algorithmica/ChitnisEM17,DBLP:conf/stoc/CurticapeanDM17,DBLP:conf/soda/FockeMR22,DBLP:journals/siamcomp/OkrasaR21,DBLP:journals/corr/abs-2312-03859}. We also point out that among all target graphs $H$, a lot of attention received odd cycles~\cite{DBLP:journals/jgt/Gerards88,DBLP:journals/dam/BeaudouFN22,DBLP:journals/combinatorica/EbsenS20,DBLP:journals/dm/SparlZ04,DBLP:journals/jct/Lai89}. Note that the cycle on $5$ vertices is the smallest graph $H$ such that the $\homo{H}$ problem is not equivalent to graph coloring. 

\paragraph{Our contribution.}
In this paper we consider the $\textsc{(L)Hom}(C_{2k+1})$ problem on bounded-diameter graphs, where $C_{2k+1}$ denotes the cycle on $2k+1$ vertices. Note that for $k=1$, we have $C_3$, so this problem is equivalent to $3$-\textsc{Coloring}. In this work we consider all other values of $k$.
Our first result is a polynomial-time algorithm for $\homo{C_{2k+1}}$ on diameter-$(k+1)$ graphs.

\begin{restatable}{theorem}{thmpoly}\label{thm:poly}
 Let $k\geq 2$. Then $\lhomo{C_{2k+1}}$ can be solved in polynomial time on diameter-$(k+1)$ graphs.
\end{restatable}

Note that such a result for $k=1$ would yield a polynomial-time algorithm for $3$-\textsc{Coloring} on diameter-$2$ graphs. Let us discuss the crucial points where this algorithm cannot be applied directly for $k=1$. The first property, which holds for every cycle except $C_3$ and $C_6$, is that if for some set $S$ of vertices, any two of them have a common neighbor, then there is a vertex that is a common neighbor of all vertices of $S$. Furthermore, for every cycle of length at least $5$ except $C_6$, for a set $S$ of $3$ vertices, every vertex of $S$ has a private neighbor with respect to $S$, i.e., a neighbor that is non-adjacent to any other vertex of $S$.

We first show that for an instance of $\lhomo{C_{2k+1}}$, for each vertex $v$ we can deduce the set of vertices it can be mapped to and bound the size of each list by $3$.
The properties discussed above allow us to encode coloring of a vertex with list of size $3$ using its neighbors with lists of size two, and such a reduced instance of a slightly more general problem (we have more constraints than just the edges, but all of them are binary) can be solved in polynomial time by reduction to $2$-\textsc{Sat}, similar to the one of Edwards~\cite{DBLP:journals/tcs/Edwards86}.

Furthermore, we provide a subexponential-time algorithm for diameter-$(k+2)$ graphs.\footnote{In the extended abstract of the paper that appeared at MFCS 2024 we stated the result also for diameter-$(k+3)$ graphs. Unfortunately, the proof does not work for diameter-$(k+3)$ graphs for every $k\geq 2$. However, it does for $k=2$, and the proof can be found in \cref{sec:subexp}.}
\begin{restatable}{theorem}{thmsubexp}\label{thm:subexp}
 Let $k\geq 2$. Then $\lhomo{C_{2k+1}}$ can be solved in time $\exp\left(\Oh\left((n\log{n})^{\frac{k+1}{k+2}}\right)\right)$ on diameter-$(k+2)$ $n$-vertex graphs.
\end{restatable}

Here the branching part of the algorithm is rather standard. The more involved part is to show that after applying braching and reduction rules we are left with an instance that can be solved in polynomial time. Similar to \cref{thm:poly}, we first analyze the lists of all vertices, and then reduce to an instance of more general problem where every vertex has list of size at most $2$.

We complement \cref{thm:poly} and \cref{thm:subexp} with the following \NP-hardness result -- since our reduction from $3$-\textsc{Sat} is linear, we also prove that the problem cannot be solved in subexponential time under the ETH.

\begin{restatable}{theorem}{thmhardness}\label{thm:hardness}
Let $k\geq 2$. The $\homo{C_{2k+1}}$ problem is \NP-hard on graphs of radius $k+1$ (and thus diameter $(2k+2)$) and cannot be solved in subexponential time, unless the ETH fails.
\end{restatable}

We summarize the current state of knowledge about the complexity of $\homo{C_{2k+1}}$ in \cref{tab:complexity}.

\begin{center}
\begin{table}[]
    \centering

\begin{tabular}{ |c|ccccccccccc| } 
  \hline
  $k$ $\setminus$ $\diam$  & 2 & 3 & 4 & 5 & 6 & 7 & 8 & 9 & 10 & 11 & $\geq 12$ \\ 
  \hline
  1 &  \cellcolor{cyan!25}   & \cellcolor{cyan!25}  & \cellcolor{red!25}   & \cellcolor{red!25}  & \cellcolor{red!25}  & \cellcolor{red!25}  & \cellcolor{red!25}  & \cellcolor{red!25}  & \cellcolor{red!25}  & \cellcolor{red!25}  & \cellcolor{red!25} \\ 
  %\hline
  2 &  \cellcolor{green!25}  & \cellcolor{green!25} & \cellcolor{cyan!25}  & \cellcolor{cyan!25}  & \cellcolor{red!25}   & \cellcolor{red!25}  & \cellcolor{red!25}  & \cellcolor{red!25}  & \cellcolor{red!25}  & \cellcolor{red!25}  & \cellcolor{red!25} \\ 
  %\hline
  3 & \cellcolor{green!25}  & \cellcolor{green!25}  & \cellcolor{green!25} & \cellcolor{cyan!25}   & ? & ? &  \cellcolor{red!25}  & \cellcolor{red!25}  & \cellcolor{red!25}  & \cellcolor{red!25}  & \cellcolor{red!25} \\ 
  %\hline
  4  & \cellcolor{green!25}  & \cellcolor{green!25}  & \cellcolor{green!25}  & \cellcolor{green!25}  & \cellcolor{cyan!25}  & ? & ? & ? & \cellcolor{red!25} & \cellcolor{red!25}  & \cellcolor{red!25} \\ 
  %\hline
  $\geq 5$  & \cellcolor{green!25}  & \cellcolor{green!25}  & \cellcolor{green!25}  & \cellcolor{green!25}  & \cellcolor{green!25}  & \cellcolor{cyan!25}   & ? & ? & ? & ? & \cellcolor{red!25} \\ 
  \hline
\end{tabular}
\caption{Complexity of $\homo{C_{2k+1}}$ on bounded-diameter graphs. The color in the cell $(k,d)$ denotes that $\homo{C_{2k+1}}$ on diameter-$d$, respectively, green -- is polynomial-time solvable, blue -- can be solved in subexponential time, and red -- cannot be solved in subexponential time, assuming the ETH.  The first row is based on~\cite{DBLP:journals/algorithmica/MertziosS16,DBLP:journals/siamdm/DebskiPR22}. The rows for $k\geq 2$ are filled due to \cref{thm:poly}, \cref{thm:subexp}, \cref{thm:hardness}, and \cref{thm:c5}.}\label{tab:complexity}
\end{table}
\end{center}

The next direction we study in the paper is the following. Instead of considering larger odd cycles and apropriate diameter, we focus on diameter-$2$ input graphs, but we change the target graph to some arbitrary $H$.
Note that it only makes sense to consider graphs $H$ of diameter-$2$, since the homomorphic image of a diameter-$2$ graph has to induce a diameter-$2$ subgraph.
We provide a polynomial-time algorithm for triangle-free target graphs.
We point out that the class of triangle-free diameter-$2$ graphs is still very rich, for instance, contains all Mycielski graphs.

\begin{restatable}{theorem}{thmtriangle}\label{thm:triangle-free}
    Let $H$ be a triangle-free graph. Then the $\lhomo{H}$ problem can be solved in polynomial time on diameter-$2$ graphs.
\end{restatable}

Finally, we point out that all upper bounds, i.e., \cref{thm:triangle-free}, \cref{thm:poly}, and \cref{thm:subexp} hold in the more general list setting, while the lower bound from \cref{thm:hardness} holds already in the non-list setting.
\section{Preliminaries}
For a vertex $v$, by $N_G(v)$ we denote the neighborhood of $v$ in $G$. 
By $\dist_G(u,v)$ we denote the length (the number of edges) of a shortest $u$-$v$ path in $G$.
For a positive integer $d$, by $N^{\leq d}_G(v)$ we denote the set of all vertices $u\in V(G)$ such that $\dist_G(u,v)\leq d$.
If $G$ is clear from the contex, we omit the subscript $G$ and simply write $N(v)$, $N^{\leq d}(v)$, and $\dist(u,v)$.
The \emph{diameter} of $G$, denoted by $\diam(G)$, is the maximum $\dist(u,v)$ over all pairs of vertices $u,v\in V(G)$. We say that $G$ is a \emph{diameter-$d$} graph if $\diam(G)\leq d$. 
The \emph{radius} of $G$ is the minimum integer $r$ such that there is a vertex $z\in V(G)$ such that for every $v\in V(G)$, it holds that $\dist(v,z)\leq r$. Note that we always have $\mathrm{radius}(G)\leq \diam(G)\leq 2\cdot \mathrm{radius}(G)$.
By $[n]$ we denote the set $\{1,2,\ldots,n\}$ and by $[n]_0$ we denote $\{0,1,\ldots,n\}$. Throughout this paper all graphs we consider are simple, i.e., no loops, no multiple edges.
By $C_n$ and $P_n$ we denote, respectively, the cycle on $n$ vertices and the path on $n$ vertices.

\paragraph{Homomorphisms.} For graphs $G,H$, a homomorphism from $G$ to $H$ is an edge-preserving mapping $\vphi: V(G)\to V(H)$, i.e., for every $uv\in E(G)$, it holds $\vphi(u)\vphi(v)\in E(H)$. For fixed $H$, called \emph{target}, in the homomorphism problem, denoted by $\homo{H}$, we are given a graph $G$, and we have to determine whether there exists a homomorphism from $G$ to $H$. In the list homomorphism problem, denoted by $\lhomo{H}$, $G$ is given along with lists $L: V(G)\to 2^{V(H)}$, and we have to determine if there is a homomorphism $\vphi$ from $G$ to $H$ which additionally respects lists, i.e., for every $v\in V(G)$ it holds $\vphi(v)\in L(v)$. We will write $\vphi: G \to H$ (resp. $\vphi: (G,L) \to H$) if $\vphi$ is a (list) homomorphism from $G$ to $H$, and $G\to H$ (resp. $(G,L)\to H$) to indicate that such a (list) homomorphism exists. Since the graph homomorphism problem generalizes graph coloring we will often refer to homomorphism as coloring and to vertices of $H$ as colors.

\paragraph{Cycles.} Whenever $C_{2k+1}$ is the target graph, we will denote its vertex set by $[2k]_0$, unless stated explicitly otherwise. Moreover, whenever we refer to the vertices of the $(2k+1)$-cycle, i.e., cycle on $2k+1$ vertices, by $+$ and $-$ we denote respectively the addition and the subtraction modulo $2k+1$.

\paragraph{Lists.} For an instance $(G,L)$ of $\lhomo{C_{2k+1}}$, by $V_i$ we denote the set of vertices $v$ of $G$ such that $|L(v)|=i$. Sometimes we will refer to vertices of $V_1$ as \emph{precolored vertices}. We also define $V_{\geq i}=\bigcup_{j\geq i} V_j$.
We say that a list $L(v)$ is \emph{of type $(\ell_1,\ldots,\ell_r)$} if $|L(v)|=r+1$ and its vertices can be ordered $c_0,\ldots,c_r$ so that for every $i\in[r-1]_0$, we have that $c_{i+1}=c_i+\ell_i$. For example, for $k\geq 4$, one of the types of the list $\{1,4,6,7\}$ is $(3,2,1)$.

\paragraph{Binary CSP and $2$-\textsc{Sat}.} For a given set (domain) $D$, in the \textsc{Binary Constraint Satisfaction} problem (\textsc{BCSP}) we are given a set $V$ of variables, list function $L: V\to 2^D$, and constraint function $C: V\times V \to 2^{D\times D}$.
The task is to determine whether there exists an assignment $f: V\to D$ such that for every $v\in V$, we have $f(v)\in L(v)$ and for every pair $(u,v)\in V\times V$, we have $(f(u),f(v))\in C(u,v)$. Clearly, any instance of $\lhomo{H}$ can be seen as an instance of $\textsc{BCSP}$, where the domain $D$ is $V(H)$, list function remains the same, and for every edge $uv\in E(G)$ we set $C(u,v)=\{(x,y) \ | \ xy\in E(H)\}$ and for every $uv\notin E(G)$ we set $C(u,v)=V(H)\times V(H)$. We will denote by $\BCSP(H,G,L)$ the instance of $\textsc{BCSP}$ corresponding to the instance $(G,L)$ of $\lhomo{H}$.
Standard approach of Edwards~\cite{DBLP:journals/tcs/Edwards86} with a reduction to 2-\textsc{Sat} implies that in polynomial time we can solve an instance of $\textsc{BCSP}$ with all list of size at most two. 

\begin{theorem}[Edwards~\cite{DBLP:journals/tcs/Edwards86}]\label{lem:red-2-sat}
    Let $(V,L,C)$ be an instance of $\textsc{BCSP}$ over the domain $D$. Assume that for every $v\in V$ it holds $|L(v)|\leq 2$. Then we can solve the instance $(V,L,C)$ in polynomial time.
\end{theorem}

%\begin{proof}
%We first do preprocessing on $(V,L,C)$ to ensure that for every $v\in V$, we have $|L(v)|=2$. If for some $v$ we have $L(v)=\emptyset$, then $(V,L,C)$ clearly is a no-instance. Suppose there is some $v\in V$ such that $L(v)=\{a\}$. Then we can remove $v$ from $V$ and for every $u\in V$, $b\in L(u)$, if $(a,b)\notin C(v,u)$ or $(b,a)\notin C(u,v)$, then we can remove $b$ from $L(u)$. So from now on we can assume that all lists have size two.

%We start creating $\phi$ by introducing two variables $x_{v,1}$ and $x_{v,2}$ for every $v\in V$. Moreover, for every $v\in V$, we order the elements on its list arbitrarily. The variables $x_{v,1}$, $x_{v,2}$ correspond to mapping $v$ either to the first or to the second element from its list.

%First we add to $\phi$ clauses that ensure us that $v$ is mapped precisely to one element, i.e., we add $(x_{v,1}\lor x_{v,2})$ and $(\neg x_{v,1}\lor \neg x_{v,2})$.

%It remains to add to $\phi$ clauses that ensure us that all constraints are satisfied. Let $(u,v)\in V\times V$. For every $(a,b)\in L(u)\times L(v)\setminus C(u,v)$ such that $a$ is $i$th in $L(u)$ and $b$ is $j$th in $L(v)$, we add clause $(\neg x_{u,i} \lor \neg x_{v,j})$. This completes the construction of $\phi$. 

%It is straightforward to verify that $\phi$ is satisfiable if and only if $(V,L,C)$ is a yes-instance of $\textsc{BCSP}$.
%\end{proof}

We will also use the following result of Feder, Hell, and Huang~\cite{DBLP:journals/combinatorica/FederHH99}.

\begin{lemma}[~\cite{DBLP:journals/combinatorica/FederHH99}]\label{lem:paths-poly}
Let $t$ be a positive integer. Then every instance of $\lhomo{P_t}$ can be solved in polynomial time.
\end{lemma}

\section{Reduction rules and basic observations}
In this section we define reduction rules and show some basic observations.

\paragraph{Reduction rules.} Let $H$ be a graph and let $(G,L)$ be an instance of $\lhomo{H}$. We define the following reduction rules.
\begin{enumerate}[(R1)]
\item If $H=C_{2k+1}$ and $G$ contains an odd cycle of length at most $2k-1$, then return NO.\label{red:cycle-short}
\item If $H=C_{2k+1}$ and in $G$ there are two $(2k+1)$-cycles with consecutive vertices respectively $c_0,\ldots,c_{2k}$ and $c'_0,\ldots,c'_{2k}$ and such that $c_0=c_0'$ and $c_i=c_j'$ for some $i,j\neq 0$, then a) if $i=j$, then identify $c_\ell$ with $c'_\ell$ for every $\ell\in [2k]$, b) if $i=-j$, then contract $c_\ell$ with $c'_{(-\ell)}$ for every $\ell\in [2k]$, c) otherwise return NO.\label{red:two-cycles}
\item For every edge $uv$, if there is $x\in L(u)$ such that $N_H(x)\cap L(v)=\emptyset$, then remove $x$ from $L(u)$.\label{red:edges}
\item If there is $v\in V(G)$ such that $L(v)=\emptyset$, then return NO.\label{red:empty}
\item For every $v\in V(G)$, if there are $x,y\in L(v)$ such that for every $u\in N_G(v)$, we have $N_H(x)\cap L(u)\subseteq N_H(y)\cap L(u)$, then remove $x$ from $L(v)$.\label{red:comp}
\item For a vertex $x\in V(H)$, and vertices $u,v\in V(G)$ such that $L(u)=L(v)=\{x\}$, if $uv\in E(G)$, then return NO, otherwise identify $u$ with $v$.\label{red:identify}

% \item reducing ``comparable'' vertices
\end{enumerate}

Clearly, each of the above reduction rules can be applied in polynomial time. Let us verify that the reduction rules are safe.

\begin{lemma}\label{lem:reduction-rules}
    After applying each reduction rule to an instance $(G,L)$ of~$\lhomo{H}$, we obtain an equivalent instance.
\end{lemma}
\begin{proof}
    First, any odd cycle cannot be mapped to a larger odd cycle, so the reduction rule (R1) is safe. Furthermore, for any $(2k+1)$-cycle, in any list homomorphism to $C_{2k+1}$, its consecutive vertices have to be mapped to consecutive vertices of $C_{2k+1}$. Let $c_0,\ldots,c_{2k}$ and $c'_0,\ldots,c'_{2k}$ be the vertices of two $(2k+1)$-cycles such that $c_0=c_0'$ and $c_i=c_j'$ for some $i,j\neq 0$.  Suppose we are dealing with a yes-instance and let $\vphi: (G,L)\to C_{2k+1}$. Without loss of generality assume that $\vphi(c_0)=\vphi(c_0')=0$ and $\vphi(c_i)=\vphi(c_j')=i$. Then $\vphi(c_s)=s$ for every $s\in [2k]_0$. Moreover, either $j=i$ or $j=-i$, and $\vphi(c_s')=s$ for every $s\in [2k]_0$ in the first case, or $\vphi(c_s')=-s$ for every $s\in [2k]_0$ in the second case. Therefore, the reduction rule (R2) is safe. 

    Let $uv\in E(G)$ be such that there is $x\in L(u)$ such that $N(x)\cap L(v)=\emptyset$. Suppose that there is a list homomorphism $\vphi: (G,L)\to C_{2k+1}$ such that $\vphi(u)=x$. Then $v$ must be mapped to a vertex from $N(x)\cap L(v)=\emptyset$, a contradiction. Thus we can safely remove $x$ from $L(u)$, and (R3) is safe. Clearly, if any list of a vertex is an empty set, then we are dealing with a no-instance and thus (R4) is safe. Finally, assume there is $v\in V(G)$ and $x,y\in L(v)$ such that for every $u\in N_G(v)$, it holds $N_H(x)\cap L(u)\subseteq N_H(y)\cap L(u)$, and suppose there is a list homomorphism $\vphi: (G,L)\to H$ such that $\vphi(v)=x$. Then $\vphi'$ defined so that $\vphi'(v)=y$ and $\vphi'(w)=\vphi(w)$ for $w\in V(G)\setminus \{v\}$ is also a list homomorphism $(G,L)\to H$. Therefore, (R5) is safe. If two vertices have the same one-element list, then they must be mapped to the same vertex. Since we only consider loopless graphs $H$, adjacent vertices of $G$ cannot be mapped to the same vertex. Therefore, if two vertices $u,v$ of $G$ have lists $L(v)=L(u)=\{x\}$ for some $x\in V(H)$, then if $uv\in E(G)$ we are dealing with a no-instance. Otherwise, we can identify $u$ and $v$ and thus (R6) is safe.
\end{proof}

We call an instance $(G,L)$ \emph{reduced} if none of the reduction rules (R1)--(R6) can be applied to $(G,L)$.

In the following observation we describe the lists of vertices that are at some small distance of a precolored vertex.

\begin{lemma}\label{obs:lists-dist}
    Let $k\geq 1$ and let $(G,L)$ be a reduced instance of $\lhomo{C_{2k+1}}$. Let $u\in V(G)$ be such that $L(u)=\{i\}$ and let $v\in V(G)$ be such that $\dist(u,v)=d$. Then
    \begin{enumerate}[a)]
        \item $L(v)\subseteq \{i-d,i-d+2,\ldots,i-2,i,i+2,\ldots,i+d-2,i+d\}$ if $d$ is even,
        \item $L(v)\subseteq \{i-d,i-d+2,\ldots,i-1,i+1,\ldots,i+d-2,i+d\}$ if $d$ is odd.
    \end{enumerate}
In particular, if $d\leq k-1$, then $L(v)$ is an independent set.
\end{lemma}

\begin{proof}
 Let $P$ be a shortest $u$-$v$ path such that the consecutive vertices of $P$ are $u=p_0,p_1,\ldots,p_d=v$. We have $L(p_0)=\{i\}$. Since the reduction rule \ref{red:edges} cannot be applied for $p_0p_1$, we must have $L(p_1)\subseteq \{i-1,i+1\}$. Suppose now that there is $j\in [d]$ such that:
\[
L(p_{j})\subseteq \{i-j,i-j+2,\ldots,i-2,i,i+2,\ldots,i+j-2,i+j\},
\]

if $j$ is even, and 
\[
L(p_{j})\subseteq \{i-j,i-j+2,\ldots,i-1,i+1,\ldots,i+j-2,i+j\},
\]
if $j$ is odd.

Then for $p_{j+1}$, which is adjacent to $p_{j}$ and thus each vertex of $L(p_{j+1})$ must be adjacent to some vertex of $L(p_{j})$, we have
\[
L(p_{j+1})\subseteq \{i-j-1,i-j+1,\ldots,i-1,i+1,\ldots,i+j-1,i+j+1\},
\]
if $j$ is even, and
\[
L(p_{j+1})\subseteq \{i-j-1,i-j+1,\ldots,i-2,i,i+2,\ldots,i+j-1,i+j+1\},
\]
if $j$ is odd. By the principle of induction the theorem follows.
\end{proof}

The next observation immediately follows from \cref{obs:lists-dist}.

\begin{lemma}\label{obs:lists-dist-two}
      Let $k\geq 1$ and let $(G,L)$ be a reduced instance of $\lhomo{C_{2k+1}}$. Let $u,w\in V(G)$ be such that $L(u)=\{i\}$ and $L(w)=\{i+1\}$. Let $v\in V(G)$ be such that $\dist(u,v)=\dist(w,v)=k+\ell$ for some $\ell\geq 0$. Then $L(v)\subseteq \{i+k-\ell+1,i+k-\ell+2,\ldots,i+k+\ell+1\}$.
\end{lemma}

In the following lemma we show that for a partial mapping $\vphi: V(G)\to [2k]_0$ for $k\geq 2$, for some vertex $v\in V(G)$, if every pair $(a,b)$ of its neighbors is precolored so that $\vphi(a)$ and $\vphi(b)$ have a common neighbor in $L(v)$, then $\vphi$ can be extended to $v$ so that it preserves the edges containing $v$.

\begin{lemma}\label{obs:neigh-2-3}
    Let $k\geq 2$, let $(G,L)$ be an instance of $\lhomo{C_{2k+1}}$ and let $v\in V(G)$. Let $\vphi: N(v)\to [2k]_0$ be a mapping such that for every $u,w\in N(v)$, we have that $N_{C_{2k+1}}(\vphi(u))\cap N_{C_{2k+1}}(\vphi(w))\cap L(v) \neq\emptyset$. Then $\bigcap_{u\in N(v)}N_{C_{2k+1}}(\vphi(u)) \cap L(v)\neq \emptyset$.
\end{lemma}

\begin{proof}
 Define $A=\{\vphi(u) \ | \ u\in N(v)\}$. If $|A|\leq 2$, then the statement clearly follows. We will show that this is the only case. So suppose that $|A|\geq 3$. If two distinct vertices of $C_{2k+1}$ for $k\geq 2$ have a common neighbor, then they must be at distance exactly two. Without loss of generality, let $0,2\in A$ and $1\in L(v)$. Moreover, let $i\in A\setminus\{0,2\}$. By assumption $N_{C_{2k+1}}(i)\cap N_{C_{2k+1}}(0)\neq \emptyset$, so $i=2k-1$. On the other hand, $N_{C_{2k+1}}(i)\cap N_{C_{2k+1}}(2)\neq \emptyset$, so $i=4$. Thus $2k-1=4$, a contradiction.
\end{proof}

In the next lemma we show that for an odd cycle $C$ and a vertex $v$ there is at least one pair of consecutive vertices of $C$ with equal distances to $v$.

\begin{lemma}\label{obs:cycle-dist}
    Let $G$ be a connected graph, let $C$ be a cycle in $G$ with consecutive vertices $c_0,\ldots,c_{2k}$ (indices computed modulo $2k+1$), and let $v\in V(G)\setminus V(C)$. Then there is $i\in [2k]_0$ such that $\dist(v,c_i)=\dist(v,c_{i+1})$.
\end{lemma}
\begin{proof}
    First observe, that for all $i$, we have $|\dist(v,c_i)-\dist(v,c_{i+1})|\leq 1$ since $c_ic_{i+1}\in E(G)$. Therefore, going around the cycle, the distance from $v$ to $c_i$ can increase by 1, decrease by 1, or remain the same. Since we have to end up with the same value at the end and the length of the cycle is odd, there is at least one pair of consecutive vertices $c_i,c_{i+1}$ such that $\dist(v,c_i)=\dist(v,c_{i+1})$.
\end{proof}

\section{Polynomial-time algorithms}\label{sec:poly}
In this section we prove \cref{thm:poly}.

\thmpoly*

\begin{proof}
Let $(G,L)$ be an instance of $\lhomo{C_{2k+1}}$ such that $G$ has diameter at most $k+1$. First, for every $i\in [2k]_0$, we check whether there is a list homomorphism $\vphi:(G,L)\to C_{2k+1}$ such that no vertex is mapped to $i$. So we look for a list homomorphism to a path which can be done in polynomial time by \cref{lem:paths-poly}.
    
If there is no such a homomorphism, then we know that all colors have to be used and thus we guess $2k+1$ vertices that will be mapped to distinct vertices of $C_{2k+1}$. Let $c_0,\ldots,c_{2k}$ be the vertices such that $c_i$ is precolored with $i$. We check whether such a partial assignment respects the lists and satisfies the edges with both endpoints precolored. Moreover, for $i\in [2k]_0$, if $c_i,c_{i+1}$ are non-adjacent, we add the edge $c_ic_{i+1}$ -- this operation is safe as $c_i,c_{i+1}$ are precolored with consecutive vertices of $C_{2k+1}$ and adding an edge does not increase the diameter. Furthermore, we exhaustively apply the reduction rules. 
%Finally, we check if $G$ contains an odd cycle of length at most $2k-1$, and if so, then we reject the instance -- recall that by \cref{obs:smaller-odd-cycle}, there is no homomorphism from $G$ to $C_{2k+1}$.

So since now we can assume that the instance $(G,L)$ is reduced. Let us analyze $(G,L)$.

Observe that the vertices $c_0,\ldots,c_{2k}$ induce a $(2k+1)$-cycle $C$. Suppose there is a vertex $v$ that is not on $C$. By \cref{obs:cycle-dist}, there is $i\in [2k]_0$ such that $\dist(v,c_i)=\dist(v,c_{i+1})=:\ell$. 

First, we show that we cannot have $\ell\leq k$.
Suppose otherwise. Let $P_1,P_2$ be shortest $v$-$c_i$-, and $v$-$c_{i+1}$-paths, respectively. Let $u$ be the their last common vertex (it cannot be $c_i$ or $c_{i+1}$ as the distances are the same and $c_i,c_{i+1}$ are adjacent). Note that since $P_1$, $P_2$ are shortest, the $u$-$c_i$-path $P_1'$ obtained from $P_1$ and the $u$-$c_{i+1}$-path $P_2'$ obtained from $P_2$ have the same length. Therefore we can construct a cycle by taking $P_1', P_2'$ and the edge $c_ic_{i+1}$. The length of this cycle is odd, and it is at most $2k+1$, which is a contradiction with the fact that none of the reduction rules \ref{red:cycle-short}, \ref{red:two-cycles} cannot be applied.

%Now suppose that $\ell=k$. By \cref{obs:lists-dist-two}, we have $L(v)\subseteq \{i+k+1\}$. If $L(v)=\emptyset$, then the reduction rule \ref{red:empty} would return NO, a contradiction.
%If $L(v)= \{i+k+1\}$, then \ref{red:identify} would identify $v$ with $c_{i+k+1}$ or would return NO, a contradiction.
Therefore we cannot have $\ell\leq k$, and thus, since $\diam(G)\leq k+1$, we have $\ell=k+1$. 

By \cref{obs:lists-dist-two}, we have that $L(v)\subseteq \{i+k, i+k+1, i+k+2\}$.
Since $v$ is an arbitrary vertex outside $C$, we can conclude that all lists of our instance have size at most 3.
Moreover, each vertex of $V_3$ (recall that by $V_r$ we denote the set of vertices of $V(G)$ with lists of size $r$) has list of type $(1,1)$.
Furthermore, since \ref{red:edges} cannot be applied, for a vertex with list $\{j,j+1,j+2\}$, the possible lists of its neighbors in $G[V_3]$ are then $\{j-1,j,j+1\}$, $\{j,j+1,j+2\}$, and $\{j+1,j+2,j+3\}$.

For a list $\{j-1,j,j+1\}$ of type $(1,1)$, we will call $j$ the \emph{middle vertex} of $\{j-1,j,j+1\}$. For a homomorphism $\vphi$, we will say that a vertex $v\in V_3$ is \emph{$\vphi$-middle}, if $\vphi$ maps $v$ to the middle vertex of its list.

Now consider a connected component $S$ of $G[V_3]$, let $v\in V(S)$, and let $L(v)=\{j-1,j,j+1\}$. The following claim is straightforward.

\setcounter{theorem}{1}
\begin{claim}\label{claim:middle}
Suppose there is a list homomorphism $\vphi: (S,L)\to C_{2k+1}$. Then

\begin{enumerate}[(1.)]
\item if $v$ is $\vphi$-middle, then any $u\in N_S(v)$ with list $\{j-1,j,j+1\}$ cannot be $\vphi$-middle, and every $w\in N_S(v)$ with list $\{j-2,j-1,j\}$ or $\{j,j+1,j+2\}$ has to be $\vphi$-middle,
\item if $v$ is not $\vphi$-middle, then every $u\in N_S(v)$ with list $\{j-1,j,j+1\}$ has to be $\vphi$-middle, and any $w\in N_S(v)$ with list $\{j-2,j-1,j\}$ or $\{j,j+1,j+2\}$ cannot be $\vphi$-middle.
\end{enumerate}
\end{claim}

%for any vertex $v\in V(S)$ with list $\{j-1,j,j+1\}$, for any homomorphism $\vphi$, if $v$ is $\vphi$-middle ($\vphi(v)=j$), then any neighbor $u$ of $v$ with list $\{j-1,j,j+1\}$ cannot be $\vphi$-middle ($\vphi(u)\neq j$), and every neighbor $w$ with list $\{j-2,j-1,j\}$ or $\{j,j+1,j+2\}$ has to be $\vphi$-middle ($\vphi(w)=j+1$ or $\vphi(w)=j-1$). 
%Similarly, if $v$ is not $\vphi$-middle ($\vphi(v)\in\{j-1,j+1\}$), then every neighbor $u$ of $v$ with list $\{j-1,j,j+1\}$ has to be $\vphi$-middle ($\vphi(u)=j$), and any neighbor $w$ with list $\{j-2,j-1,j\}$ or $\{j,j+1,j+2\}$ cannot be $\vphi$-middle ($\vphi(w)\neq j+1$ and $\vphi(w)\neq j-1$).
Thus deciding if one vertex of $S$ is $\vphi$-middle, already determines for every vertex of $S$ if it is $\vphi$-middle or not. It is described more formally in the following claim.

\begin{claim}\label{claim}
In polynomial time we can either (1) construct a partition $(U_1,U_2)$ of $V(S)$ ($U_1$, $U_2$ might be empty) such that for every list homomorphism $\vphi: (S,L)\to C_{2k+1}$, either all vertices of $U_1$ are $\vphi$-middle and no vertex of $U_2$ is $\vphi$-middle, or all vertices of $U_2$ are $\vphi$-middle and no vertex of $U_1$ is $\vphi$-middle, or (2) conclude that we are dealing with a no-instance.
\end{claim}

\begin{claimproof}
    Fix $v\in V(S)$. We start with $U_1:=\{v\}$ and $U_2:=\emptyset$. Let $U=U_1\cup U_2$ and $j=|U|+1$. We set $v_1:=v$. As long as $j\leq |S|$, we proceed as follows. 
    We set $v_j$ to be a vertex in $N(U_1\cup U_2)$ -- since $U_1$ is non-empty, $j\leq |S|$, and $S$ is connected, such $v_j$ always exists. Moreover, let $N_j^1$ be the set of neighbors of $v_j$ with list $L(v_j)$ and let $N_j^2=N(v_j)\setminus N_j^1$. 
    
If (a) $N_j^1\cap U\subseteq U_1$ and $N_j^2\cap U\subseteq U_2$, then add $v_j$ to $U_2$. If (b) $N_j^1\cap U\subseteq U_2$ and $N_j^2\cap U\subseteq U_1$, then add $v_j$ to $U_1$. 
    If none of the two cases holds, return NO. In the first two cases we repeat the procedure. Clearly, the above procedure can be performed in polynomial time. 

    Let us verify the correctness. We will show the following.
    
    \smallskip
    
    ($\star$) For every $j\in [|S|]$, for every list homomorphism $\vphi: (S[\{v_1,\ldots,v_j\}],L)\to C_{2k+1}$, it holds that either (i) all vertices of $U_1\cap \{v_1,\ldots,v_j\}$ are $\vphi$-middle and no vertex of $U_2\cap \{v_1,\ldots,v_j\}$ is $\vphi$-middle, or (ii) all vertices of $U_2\cap \{v_1,\ldots,v_j\}$ are $\vphi$-middle and no vertex of $U_1\cap \{v_1,\ldots,v_j\}$ is $\vphi$-middle.
    
    \smallskip
    
    We prove it by induction on $j$. For $j=1$, we have $U_1=\{v\}$ and $U_2=\emptyset$, so the statement clearly follows. So now assume that the statement is true for some $i\in [|S|-1]$. Suppose there is a list homomorphism $\vphi: (S[\{v_1,\ldots,v_{i+1}\}],L)\to C_{2k+1}$. By inductive assumption, for $\vphi|_{\{v_1,\ldots,v_i\}}$, either (i) every vertex of $\{v_1,\ldots,v_i\}\cap U_1$ is $\vphi$-middle and no vertex of $\{v_1,\ldots,v_i\}\cap U_2$ is $\vphi$-middle or (ii) every vertex of $\{v_1,\ldots,v_i\}\cap U_2$ is $\vphi$-middle and no vertex of $\{v_1,\ldots,v_i\}\cap U_1$ is $\vphi$-middle.
    
\paragraph{Case 1: $v_{i+1}$ is $\vphi$-middle.}
By \cref{claim:middle}~(1.), all neighbors of $v_{i+1}$ in $\{v_1,\ldots,v_i\}$ with lists different than $L(v_{i+1})$ have to be $\vphi$-middle, so in case (i) all such neigbors have to be in $U_1$, and in case (ii) all such neighbors have to be in $U_2$. Furthermore, by \cref{claim:middle}~(1.), all neighbors of $v$ in $\{v_1,\ldots,v_i\}$ with list $L(v_{i+1})$ cannot be $\vphi$-middle, so in case (i) all such neigbors have to be in $U_2$, and in case (ii) all such neighbors have to be in $U_1$. Recall, that if $N_j^1\cap U\subseteq U_2$ and $N_j^2\cap U\subseteq U_1$ which is the case in (i), we added $v_{i+1}$ to $U_1$, so the claim follows. Similarly, if $N_j^1\cap U\subseteq U_1$ and $N_j^2\cap U\subseteq U_2$ which is the case in (ii), we added $v_{i+1}$ to $U_2$, and the claim also follows in this case.
    
\paragraph{Case 2: $v_{i+1}$ is not $\vphi$-middle.} By \cref{claim:middle}~(2.), any neighbor of $v_{i+1}$ in $\{v_1,\ldots,v_i\}$ with list different than $L(v_{i+1})$ cannot be $\vphi$-middle, so in case (i) all such neighbors have to be in $U_2$, and in case (ii) all such neighbors have to be in $U_1$. Furthermore, by \cref{claim:middle}~(2.), all neighbors of $v$ in $\{v_1,\ldots,v_i\}$ with list $L(v_{i+1})$ have to be $\vphi$-middle, so in case (i) all such neighbors have to be in $U_1$, and in case (ii) all such neighbors have to be in $U_2$. But then in cases (i) and (ii), respectively, we added $v_{i+1}$ to $U_2$ and $U_1$, and in each of the cases the claim follows.

This completes the proof of ($\star$).

Finally, observe that since ($\star$) is true, if $(G,L)$ is a yes-instance, then for any $j\in [|S|]$, for a vertex $v_j\in U_i$, where $\{i,i'\}=\{1,2\}$, all neighbors of $v_j$ with list $L(v_j)$ have to be in $U_{i'}$ and  all neighbors of $v_j$ with list different than $L(v_j)$ have to be in $U_{i}$. Therefore, if we returned NO, then indeed, $(G,L)$ is a no-instance. This completes the proof of the claim.
\end{claimproof}

 Therefore, for every connected component $S$ of $G[V_3]$ we solve two subinstances:
 \begin{enumerate}
     \item $I_1(S)=(S,L_1)$, where for every $v\in V(S)$ with list $\{i-1,i,i+1\}$ for some $i\in [2k]_0$, $L_1(v)=\{i\}$ if $v\in U_1$ and $L_1(v)=\{i-1,i+1\}$ if $v\in U_2$,
     \item $I_2(S)=(S,L_2)$, where for every $v\in V(S)$ with list $\{i-1,i,i+1\}$ for some $i\in [2k]_0$, $L_2(v)=\{i\}$ if $v\in U_2$ and $L_2(v)=\{i-1,i+1\}$ if $v\in U_1$.
 \end{enumerate}

Note that both subinstances have all lists of size at most two and thus can be solved in polynomial time by \cref{lem:red-2-sat}. If for some component in both cases we obtain NO, then we return NO.

\paragraph{Creating a BCSP instance.} Let $(V(G),L,C)=\BCSP(C_{2k+1},G[V_1\cup V_2],L)$ -- note that since we only consider vertices of $V_1\cup V_2$, all lists have size at most two. We will now modify the instance $(V(G),L,C)$ so that it is equivalent to the instance $(G,L)$ of $\lhomo{C_{2k+1}}$. For every $v\in V_3$ and for every pair of $u,w\in N(v)\cap V_2$, we leave in $C(u,w)$ only these pairs of vertices that have a common neighbor in $L(v)$ -- recall that by \cref{obs:neigh-2-3} this ensures us that there will be a color left for $v$.
Furthermore, for every connected component $S$ of $G[V_3]$, we add constraints according to which of the two possibilities $S$ can be properly colored (possibly $S$ can be colored in both cases) as follows. Let $\{p,p'\}=\{1,2\}$.
\begin{enumerate}
\item If there is no list homomorphism $\vphi: (S,L)\to C_{2k+1}$ such that all vertices of $U_p$ are $\vphi$-middle and all vertices of $U_{p'}$ are not $\vphi$-middle, then for $v\in V(S)$ with $L(v)=\{i-1,i,i+1\}$, if $v\in U_p$, we remove $i-1,i+1$ from the lists of neighbors of $v$, and if $v\in U_{p'}$, we remove $i$ from the lists of neighbors of $v$.

\item Moreover, for every pair $u,v\in V(S)$ with $L(u)=\{j-1,j,j+1\}$ and $L(v)=\{i-1,i,i+1\}$, for every $u'\in V_2\cap N(u)$, and for every $v'\in V_2\cap N(v)$, if $u,v\in U_p$, then we remove from $C(u',v')$ the pairs $(j,i+1)$, $(j,i-1)$, $(j-1,i)$, $(j+1,i)$, and if $v\in U_p$, $v\in U_{p'}$, then we remove from $C(u',v')$ the pairs $(j,i)$, $(j-1,i-1)$, $(j-1,i+1)$, $(j+1,i-1)$, $(j+1,i+1)$.
\end{enumerate}

Let us describe the intuition behind both types of constraints. Suppose that we have colored $G[V_1\cup V_2]$ so that all added constraints are satisfied and let $S$ be a connected of $G[V_3]$. Constraints added in 1. ensure that if one of the instances $I_1(S)$ and $I_2(S)$ is a no-instance, then the colors left for vertices of $S$ correspond only to that $I_p(S)$ for $p\in \{1,2\}$, which is a yes-instance. Furthermore, 2. ensures us that the colors left for vertices of $S$ correspond to the same instance $I_p(S)$.

This completes the construction of \textsc{BCSP} instance $(V(G),L,C)$.
Recall that all lists have size at most $2$ and thus by \cref{lem:red-2-sat} we solve $(V(G),L,C)$ in polynomial time.

\paragraph{Correctness.} It remains to show that $(V(G),L,C)$ is equivalent to $(G,L)$.
First suppose that there is a list homomorphism $\vphi: (G,L)\to C_{2k+1}$. Let us define $f=\vphi|_{V_1\cup V_2}$. Clearly $f$ respects the lists and the constraints coming from $\BCSP(C_{2k+1},G,L)$. Furthermore, since $\vphi$ is a homomorphism, for every $v\in V_3$ and every pair $u,w\in V_2$ of neighbors of $v$, it holds that $f(u)$ and $f(w)$ have a common neighbor in $L(v)$.
Suppose now that for some vertex $v$ with $L(v)=\{i-1,i,i+1\}$, for $v'\in V_2\cap N(v)$, we removed from $L(v')$ color $\vphi(v')$. Let $S$ be the connected component of $G[V_3]$ containing $v$. If $\vphi(v')=i$, then we must have $\vphi(v)\in \{i-1,i+1\}$ -- but we only removed $i$ from $L(v')$ if there was no list homomorphism from $(G[S],L)$ to $C_{2k+1}$ mapping $v$ to one of $\{i-1,i+1\}$, a contradiction. Similarly, if $\vphi(v')\in \{i-1,i+1\}$, then we must have $\vphi(v)=i$, but we only removed $i-1,i+1$ from $L(v')$ if there was no list homomorphism from $(G[S],L)$ to $C_{2k+1}$ mapping $v$ to $i$, a contradiction.

Finally, let us verify the constraints added in the last step. 
Let $S$ be a connected component of $G[V_3]$, let $u,v\in V(S)$ with $L(v)=\{i-1,i,i+1\}$ and $L(u)=\{j-1,j,j+1\}$, let $v'\in N(v)\cap V_2$ let $u'\in N(u)\cap V_2$. Suppose that $(f(u'),f(v'))\notin C(u',v')$ because of the last step. 
Let $\{p,p'\}=\{1,2\}$ and let $v\in U_p$. If $u\in U_{p}$, then either both $u,v$ are $\vphi$-middle or none of them. In the first case $\vphi(v')=i$ and $\vphi(u')=j$ and in the second case $\vphi(v')\in \{i-1,i+1\}$ and $\vphi(u')\in\{j-1,j+1\}$. Recall that for $u,v\in U_p$, we did not remove any pair from $(\{i-1,i+1\}\times \{j-1,j+1\})\cup \{(i,j)\}$, a contradiction. So suppose that $u\in U_{p'}$. Then $v$ is $\vphi$-middle if and only if $u$ is not $\vphi$-middle. 
Therefore, we either have $\vphi(v')\in \{i-1,i+1\}$ and $\vphi(u')=j$, or $\vphi(v')=i$ and $\vphi(v')=i$ and $\vphi(u')\in \{j-1,j+1\}$, but again no such a pair was removed in the last step, a contradiction.

%Suppose that $(f(u),f(w))\notin C(u,w)$ because of the last step in the construction of $(V(G),L,C)$. Then either $f(u)=i$ and $f(w)=i+2$ or $f(u)=f(w)=i$. In the first case $\vphi(v)=i+1$, which contradicts the fact that $S$ cannot be colored so that $v$ is mapped to $i+1$ and in the second case $\vphi(v)\in \{i,i+2\}$, which is also a contradiction.

So now suppose that there is $f: V(G)\to [2k]_0$ that satisfies all constraints of $(V(G),L,C)$. Define $\vphi(v)=f(v)$ for $v\in V_1\cup V_2$. Note that since $f$ satisfies all constraints, $\vphi$ is a list homomorphism on $G[V_1\cup V_2]$. We have to show that $\vphi$ can be extended to vertices of $V_3$. 
Consider a connected component $S$ of $G[V_3]$. Since the algorithm did not return NO, there is a list homomorphism $\vphi_S:(S,L)\to C_{2k+1}$. 
First, note that for any $v\in V_3$, its neighbors from $V_2$ are colored so that there is a color left on $L(v)$ for $v$.  
Moreover, recall that if $L(v)=\{i-1,i,i+1\}$, then the possible lists of neighbors of $v$ in $V_2$ are $\{i-1,i\}$ and $\{i,i+1\}$. 
Therefore, either all neighbors of $v$ in $V_2$ are colored with $i$ or with $i-1$ and $i+1$. In the first case, colors left for $v$ are both $i-1$ and $i+1$, and in the second case $i$ is left for $v$. 
Furthermore, let $\{p,p'\}=\{1,2\}$ and let $u,v\in V(S)$. If $u,v\in U_p$, then by the constraints added in 2., the colors left for $u,v$ allow both of them to be $\vphi$-middle, or both of them to be not $\vphi$-middle. Similarly, if $v\in U_p$ and $u\in U_{p'}$, then the colors left for $u,v$ allow one of them to be $\vphi$-middle and the other to be not $\vphi$-middle. Therefore, if we leave on the lists of vertices of $S$ only those colors that match $\vphi[V_1\cup V_2]$, we obtain one of the instances $I_1(S)$, $I_2(S)$, and by the constraints added in 1., we can only obtain a yes-instance. Thus we can extend $\vphi$ to all vertices of $S$. This completes the proof.
\end{proof}
\setcounter{theorem}{5}

\section{Subexponential-time algorithms}\label{sec:subexp}
In this section we prove \cref{thm:subexp}.
 
\thmsubexp*

We start with defining branching rules crucial for our algorithm. Recall that for an instance $(G,L)$ of $\lhomo{C_{2k+1}}$, by $V_\ell$ we denote the subset of $V(G)$ such that for every $v\in V_\ell$, we have $|L(v)|=\ell$, and $V_{\geq \ell}= \bigcup_{i\geq\ell} V_i$.

\paragraph{Branching rules.} Let $k\geq 2$, let $I=(G,L)$ be an instance of $\lhomo{C_{2k+1}}$, and let $d\geq\diam(G)$ -- note that for \cref{thm:subexp} it suffices to consider only $d=k+1$, but we do it more generally so we can later use it also for other values of $d$. Let $\mu=\sum_{\ell=2}^{2k+1}\ell\cdot|V_{\ell}|$. We define the following branching rules.
\begin{enumerate}[(B1)]
\item For a vertex $v\in V_{\geq 2}$ and for a color $a\in L(v)$, we branch on coloring $v$ with $a$ or not, i.e., we create two subinstances of $I$: $I_a=(G,L_a)$, $I_a'=(G,L'_a)$ such that $L_a(u)=L_a'(u)=L(u)$ for every $u\in V(G)\setminus\{v\}$, and $L_a(v)=\{a\}$ and $L_a'(v)=L(v)\setminus\{a\}$.
\item For a vertex $v$ we branch on the coloring of $N^{\leq d-1}[v]\cap V_{\geq 2}$, i.e., for every mapping $f$ of $N^{\leq d-1}[v]\cap V_{\geq 2}$ that respects the lists, we create a new subinstance $I_f=(G,L_f)$ such that $L_f(u)=L(u)$ for $u\notin N^{\leq d-1}[v]\cap V_{\geq 2}$ and $L_f(w)=\{f(w)\}$ for $w\in N^{\leq d-1}[v]\cap V_{\geq 2}$. 
\end{enumerate}

\paragraph{Algorithm \texttt{Recursion Tree}.} Let us describe an algorithm that for fixed $d$ takes an instance $(G,L)$ of $\lhomo{C_{2k+1}}$ with a fixed precolored $(2k+1)$-cycle $C$ and such that $\diam(G)\leq d$, and returns a rooted tree $\cR$ whose nodes are labelled with subinstances of $(G,L)$.
We first introduce the root $r$ of $\cR$ and we label it with $(G,L)$. 
Then for every node we proceed recursively as follows.
Let $s$ be a node labelled with an instance $(G',L')$ of $\lhomo{C_{2k+1}}$. We first exhaustively apply to $(G',L')$ the reduction rules (R1)--(R6), and if some of the reduction rules returns NO, then $s$ does not have any children. 
Otherwise, if there is a vertex $v\in V_{\geq 2}$ with at least $(\mu\log\mu)^{1/d}$ neighbors in $V_{\geq 2}$, then we will apply (B1) for $v$ and for $a\in L(v)$ chosen as follows.
If on $N_{G'[V_{\geq 2}]}(v)$ there are no lists of type $(2)$, then we take any $a\in L(v)$. Otherwise, let $S$ be the most frequent list of type $(2)$ on $N_{G'[V_{\geq 2}]}(v)$, and let $S=\{j-1,j+1\}$ for some $j\in [2k]_0$. Then we take any $a\in L(v)\setminus\{j\}$.
After application of (B1), we exhausively apply the reduction rules to each instance. Furthermore, for each instance created by (B1), we create a child node of $s$ and we label it with that instance. 

If there is no vertex $v\in V_{\geq 2}$ with at least $(\mu\log\mu)^{1/d}$ neighbors in $V_{\geq 2}$, then we apply the branching rule (B2) for some $v\in V_{\geq 2}$, we exhausively apply the reduction rules, and again for each instance created by (B2), we introduce a child node of $s$.
The choice of $v$ is not completely arbitrary. If possible, we choose $v$ so that $\dist(v,C)\geq \lceil \frac{d}{2}\rceil$ -- note that the cycle $C$ is present in all instances of $\cR$.
The nodes corresponding to instances created by (B2) are leaves, i.e., we do not recurse on the children of $s$ for which we applied (B2).   

Let us analyze the running time of the algorithm \texttt{Recursion Tree} and properties of the constructed tree $\cR$.

\begin{lemma}\label{lem:recursion-tree}
Given an instance $(G,L)$ of $\lhomo{C_{2k+1}}$ with a fixed precolored $(2k+1)$-cycle $C$ and such that $n=|G|$, $\diam(G)\leq d$, the algorithm \texttt{Recursion Tree} in time $\exp\left({\Oh((n\log{n})^{\frac{d-1}{d}})}\right)$ returns a tree $\cR$ whose nodes are labelled with subinstances of $(G,L)$, and $(G,L)$ is a yes-instance if and only if at least one subinstance corresponding to a leaf of $\cR$ is a yes-instance.
\end{lemma}
\begin{proof}
%First we show that $(G,L)$ is a yes-instance if and only if every instance corresponding to a leaf in $\cR$ is a yes-instance. It is sufficient to 
First, we show that for every node $s$ of $\cR$ the corresponding instance is a yes-instance if and only if at least one instance corresponding to a child of $s$ is a yes-instance.
Let $s$ be a node of $\cR$ and let $(G',L')$ be the corresponding instance. The algorithm \texttt{Recursion Tree} first applies reduction rules to $(G',L')$ and by \cref{lem:reduction-rules}, we obtain an equivalent instance. Furthermore, we applied to $(G',L')$ either (B1) or (B2) where the branches correspond to all possible colorings of some set of vertices, so indeed $(G',L')$ is a yes-instance if and only if at least one instance corresponding to a child of $s$ is a yes-instance. Since the root of $\cR$ is labelled with $(G,L)$, we conclude that $(G,L)$ is a yes-instance if and only if at least one instance corresponding to a leaf of $\cR$ is a yes-instance.

It remains to analyze the running time. Let $F(\mu)$ be an upper bound on the running time of \texttt{Recursion Tree} applied to an instance $(G',L')$ with $\mu=\sum_{\ell=2}^{2k+1}\ell\cdot|V_{\ell}|$ and let $\p(n)$ be a polynomial such that exhaustive application of reduction rules (R1)--(R6) to an $n$-vertex instance can be performed in time $\p(n)$. Observe that if we apply (B1) to $(G',L')$, then we obtain
\[
F(\mu)\leq F\left(\mu - \frac{(\mu\log\mu)^{1/d}}{2k+1}\right) + F(\mu-1) + 2\cdot \p(n).
\]

Indeed, let $v,a$ be, respectively, the vertex and the color to which we apply (B1) -- recall that $v$ has at least $(\mu\log\mu)^{1/d}$ neighbors in $V_{\geq 2}$ If there are no lists of type $(2)$ on $N_{G'[V_{\geq 2}]}(v)$, then in the branch where we set $L(v)=\{a\}$, after application of reduction rules, every neighbor of $v$ must have $L(v)\subseteq \{a-1,a+2\}$. If $|L(v)|\geq 2$ and $L(v)\neq \{a-1,a+1\}$, then $|L(v)|\cap \{a-1,a+1\}|< |L(v)|$. Therefore, in this case we decrease sizes of all lists on $N_{G'[V_{\geq 2}]}(v)$.
Otherwise, we chose $a\in L(v)\setminus\{j\}$, where $\{j-1,j+1\}$ is the most frequent list of type $(2)$ on $N_{G'[V_{\geq 2}]}(v)$. Since there are exactly $2k+1$ lists of type $(2)$, at least $\frac{1}{2k+1}$-fraction of $N_{G'[V_{\geq 2}]}(v)$ has list of different type than $(2)$ or has list $\{j-1,j+1\}$. Thus, for the branch where we set $L(v)=\{a\}$, the sizes of lists of at least $\frac{1}{2k+1}\cdot(\mu\log\mu)^{1/d}$ vertices decrease.  In the branch where we remove $a$ from $L(v)$, we decrease the size of $L(v)$ at least by one. In both branches we apply the reduction rules, so the desired inequality follows.

If we apply (B2) to $(G',L')$ -- recall that we stop recursing in this case -- then we obtain
\[
F(\mu)\leq (2k+1)^{ (\mu\log\mu)^{\frac{d-1}{d}}} \cdot \p(n),
\]
since we guess the coloring on $N_{G'[V_\geq 2]}^{\leq d-1}(v)$ whose size is bounded by $(\mu\log\mu)^{\frac{d-1}{d}}$ (in this case the degrees in $G'[V_{\geq 2}]$ are bounded by $(\mu\log\mu)^{1/d}$) and the number of possible colors is at most $2k+1$. 

We can conclude that $F(\mu)\leq 2^{\Oh((\mu\log\mu)^{\frac{d-1}{d}})}$ (see for example~\cite{DBLP:journals/siamdm/DebskiPR22}) which, combined with the inequality $\mu\leq (2k+1)n=\Oh(n)$, completes the proof.
\end{proof}

In the following lemma we show that we can solve every instance corresponding to a leaf of $\cR$ in polynomial time.

\begin{restatable}{lemma}{lemLeafInstance}\label{lem:leaf-instance}
Let $(G',L')$ be an instance of $\lhomo{C_{2k+1}}$ such that $\diam(G')\leq k+2$ and let $C$ be a fixed precolored $(2k+1)$-cycle. Assume that we applied the algorithm \texttt{Recursion Tree} to $(G',L')$ and let $\cR$ be the resulting recursion tree. Let $(G,L)$ be an instance corresponding to a leaf in $\cR$. Then $(G,L)$ can be solved in polynomial time.
\end{restatable}

In order to prove \cref{lem:leaf-instance}, we first prove that we can solve every instance whose lists are of special form in polynomial time.
%will use the fact that the lists present in an instance corresponding to a leaf in $\cR$ either have size at most two or are of type $(2,2)$.

\begin{lemma}\label{lem:lists-3}
Let $k\geq 2$ and let $(G,L)$ be a reduced instance of $\lhomo{C_{2k+1}}$ such that $G$ is connected and for every vertex $v\in V(G)$, the list $L(v)$ either has size at most $2$ or is of type $(2,2)$. Then $(G,L)$ can be solved in polynomial time.
\end{lemma}

\begin{proof}
%By Lemma~\ref{lem:lists-3}, all vertices of $G$ have lists of size at most $3$ and each vertex of $V_3$ has list of type $(2,2)$.
If there are no lists of size $3$, then we only have lists of size at most $2$, and thus $(G,L)$ can be solved in polynomial time by \cref{lem:red-2-sat}.
So let $v\in V_{3}$ be a vertex with list $L(v)=\{i-2,i,i+2\}$ for some $i\in [2k]_0$ and let $u$ be a neighbor of $v$.
Observe that since the reduction rule \ref{red:edges} cannot be applied, if $u\in V_2$, then $u$ must have one of the lists: $\{i-1,i+1\}$  $\{i-3,i+1\}$, $\{i-1,i+3\}$, and if $u\in V_{\geq 3}$, then $u$ must have one of the lists $\{i-3,i-1,i+1\}$, or $\{i-1,i+1,i+3\} $ (see \cref{fig:lists}).

    \begin{center}
\begin{figure}[h]
    \centering
    \begin{tikzpicture}[every node/.style={draw,circle,fill=white,inner sep=0pt,minimum size=8pt},every loop/.style={},scale=0.8]
    
\foreach \k in {0,5,10}
{    
\node (v0) at (0.5+\k,-0.5) {};
\node (v1) at (1.5+\k,-0.5) {};
\node (v2) at (2.2+\k,0) {};
\node[fill=orange] (v3) at (2.6+\k,1) {};
\node (v4) at (2.2+\k,2) {};
\node[fill=orange] (v5) at (1+\k,2.5) {};
\node (v6) at (-0.2+\k,2) {};
\node[fill=orange] (v7) at (-0.6+\k,1) {};
\node (v8) at (-0.2+\k,0) {};

\draw (v0)--(v1)--(v2)--(v3)--(v4)--(v5)--(v6)--(v7)--(v8)--(v0);

}

\foreach \k in {2.5,7.5}
{    
\foreach \j in {-4}
{
\node (v0) at (0.5+\k,-0.5+\j) {};
\node (v1) at (1.5+\k,-0.5+\j) {};
\node (v2) at (2.2+\k,0+\j) {};
\node[fill=orange] (v3) at (2.6+\k,1+\j) {};
\node[fill=blue] (v4) at (2.2+\k,2+\j) {};
\node[fill=orange] (v5) at (1+\k,2.5+\j) {};
\node[fill=blue] (v6) at (-0.2+\k,2+\j) {};
\node[fill=orange] (v7) at (-0.6+\k,1+\j) {};
\node (v8) at (-0.2+\k,0+\j) {};

\draw (v0)--(v1)--(v2)--(v3)--(v4)--(v5)--(v6)--(v7)--(v8)--(v0);
}
}

\node[draw=none,fill=blue] (a) at (2.2,2) {};

\node[draw=none,fill=blue] (a) at (-0.2,2) {};

\node[draw=none,fill=blue] (a) at (7.2,2) {};

\node[draw=none,fill=blue] (a) at (9.8,2) {};

\node[draw=none,fill=blue] (a) at (4.8,0) {};

\node[draw=none,fill=blue] (a) at (12.2,0) {};

\node[draw=none,fill=blue] (a) at (4.7,-4) {};
\node[draw=none,fill=blue] (a) at (7.3,-4) {};

 \end{tikzpicture}
    \caption{\label{fig:lists} Case $k=4$. Orange vertices denote a list of some vertex $v\in V_3$, blue vertices denote all possible lists of a neighbor $u$ of $v$ when \ref{red:edges} cannot be applied, i.e., every vertex of $L(u)$ is a neighbor of a vertex of $L(v)$ and every vertex of $L(v)$ is a neighbor of a vertex of $L(u)$.}
    \end{figure}
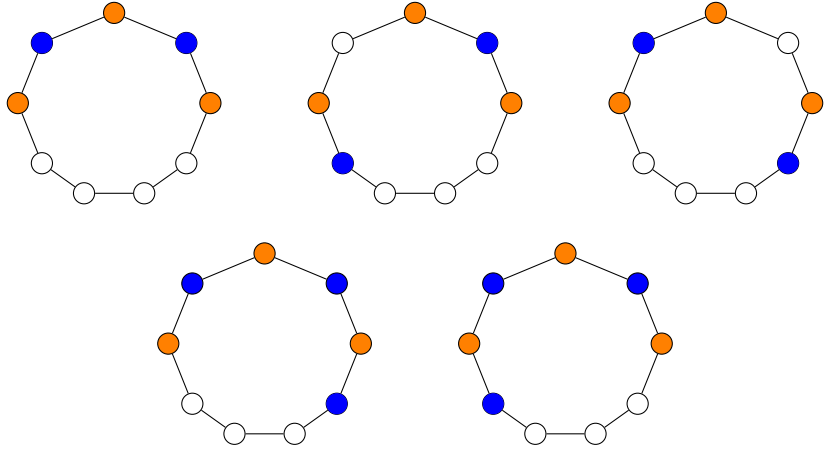
    \end{center}

If for some vertex $v$ with list $\{i-2,i,i+2\}$ there is no $u\in N(v)$ with one of lists $\{i-1,i+1\}$, $\{i-3,i+1\}$, $\{i-1,i+3\}$, then we add such a vertex $u$ to $G$ and make it adjacent to $u$. 
Note that now the diameter of $G$ might increase, but in this lemma we only need $G$ to be connected.
%but we do not care about the diameter anymore. 
Moreover, any list homomorphism on $G-u$ can be extended to $u$, since each color on $L(v)$ has a neighbor on $L(u)$.
So since now, we can assume that 

\begin{description}
  \item[($\star$)] for a vertex $v$ with list $\{i-2,i,i+2\}$, all lists $\{i-1,i+1\}$, $\{i-3,i+1\}$, $\{i-1,i+3\}$ are present on $N(v)$.
\end{description}
  
\paragraph{Constructing a \textsc{BCSP} instance.} We construct an instance of \textsc{BCSP} as follows. We start with $\BCSP(C_{2k+1},G-V_3,L)$ -- note that in this instance all lists have size at most~$2$. First, for every vertex $v\in V_3$ and for every $v',v''\in V_2\cap N(v)$, we leave in $C(v',v'')$ only such pairs that have a common neighbor in $L(v)$.
Furthermore, for every edge $uv$ with $u,v\in V_3$ such that $L(u)=\{i-2,i,i+2\}$, $L(v)=\{i-1,i+1,i+3\}$, and for every pair $u',v'\in V_2$ such that $uu',vv'\in E(G)$, we remove (if they are present) from $C(u',v')$ the following pairs: $(i-3,i+2)$, $(i-1,i+4)$, and $(i+3,i-2)$.
This completes the construction of the instance $(V,L,C)$ of \textsc{BCSP}. 

Clearly the instance $(V,L,C)$ is constructed in polynomial time. Moreover, since all lists of $(V,L,C)$ have size at most $2$, by \cref{lem:red-2-sat}, this instance can be solved in polynomial time.
    
\paragraph{Correctness.} It remains to show that the instance $(G,L)$ is equivalent to the instance $(V,L,C)$. First suppose that there is a list homomorphism $\vphi: (G,L)\to C_{2k+1}$. Consider the assignment $f=\vphi|_{V_1\cup V_2}$. We have to verify that $f$ satisfies all constraints of the instance $(V,L,C)$. Clearly $f$ satisfies all constraints coming from $\BCSP(C_{2k+1},G-V_3,L)$. Furthermore, since $\vphi$ is a homomorphism from $G$ to $C_{2k+1}$, for every $v\in V_3$ and every pair $v',v''\in N(v)\cap V_2$, we have that $f(v')$ and $f(v'')$ must have a common neighbor in $L(v)$. Finally, we have to verify that $f$ satisfies constraints that were introduced for edges of $G[V_3]$. Suppose that there are $u',v'$ such that $(f(u'),f(v'))$ was removed from $C(u',v')$ for an edge $uv$ of $G[V_3]$ such that $uu',vv'\in E(G)$. Let $L(u)=\{i-2,i,i+2\}$ and $L(v)=\{i-1,i+1,i+3\}$. Suppose first that $f(u')=i-3$ and $f(v')=i+2$. Since $\vphi$ is a list homomorphism and $uu'\in E(G)$, we must have $\vphi(u)=i-2$, and since $uv\in E(G)$ we must have $\vphi(v)=i-1$. However, $i-1$ is non-adjacent to $i+2$ and thus the edge $vv'$ cannot be properly colored, a contradiction. Now suppose that $f(u')=i-1$ and $f(v')=i+4$. Similarly, we must have $\vphi(v)=i+3$ and $\vphi(u)\in \{i-2,i\}$ so the edge $uv$ cannot be properly colored. Finally, suppose that $f(u')=i+3$ and $f(v')=i-2$. Then we must have $\vphi(u)=i+2$ and $\vphi(v)=i-1$ so again $uv$ cannot be properly colored, a contradiction.
Therefore, $f$ satisfies all constraints.

Now suppose that there is a satisfying assignment $f: V\to V(H)$ of $(V,L,C)$. Observe that if we consider $\vphi=f$ on $V_1\cup V_2$, then $\vphi$ is a list homomorphism on $G[V_1\cup V_2]$. It remains to show that $\vphi$ can be extended to vertices of $V_3$. Consider a vertex $v\in V_3$ with list $\{i-2,i,i+2\}$ for some $i\in [2k]_0$. Recall that by ($\star$), $v$ has neighbors with lists $\{i-3,i+1\}$ and $\{i-1,i+3\}$. Since $f$ satisfies the constraints, for every pair of neighbors $v',v''$ of $v$ in $V_2$, $f(v')\cap f(v'')\cap L(v)\neq \emptyset$, and thus by \cref{obs:neigh-2-3}, there is a common neighbor of the colors on $N(v)$ in $L(v)$. Furthermore, we claim that this common neighbor is unique. Indeed, if the neighbor of $v$ with list $\{i-3,i+1\}$ is colored with $i-3$, then we already know that $v$ has to be mapped to $i-2$. Otherwise, $v$ has to be mapped to one of $i,i+2$. If the neighbor of $v$ with list $\{i-1,i+3\}$ is mapped to $i+3$, then $v$ has to be mapped to $i+2$, otherwise, $v$ has to be mapped to $i$. Therefore, we extend $\vphi$ to vertices of $V_3$ in the only possible way. It remains to show that $\vphi$ respects the edges of $V_3$. Suppose not and let $uv$ be an edge which is not properly colored. Let $L(u)=\{i-2,i,i+2\}$ and $L(v)=\{i-1,i+1,i+3\}$. First suppose that $\vphi(u)=i-2$ and $\vphi(v)\in \{i+1,i+3\}$. A neighbor of $u$ with list $\{i-3,i+1\}$ has to be mapped to $i-3$ and a neighbor of $v$ with list $\{i-2,i+2\}$ has to be mapped to $i+2$, but then the mapping $f$ does not satisfy the constraints since we removed the pair $(i-3,i+2)$. So now suppose that $u$ is mapped to $i$ and $v$ is mapped to $i+3$. Then a neighbor of $v$ with list $\{i,i+4\}$ is mapped to $i+4$ and a neighbor of $u$ with list $\{i-1,i+3\}$ is mapped to $i-1$, but the pair $(i-1,i+4)$ was removed. Finally suppose that $u$ is mapped to $i+2$ and $v$ is mapped to $i-1$. Then a neighbor of $u$ with list $\{i-1,i+3\}$ is mapped to $i+3$ and a neighbor of $v$ with list $\{i-2,i+2\}$ is mapped to $i-2$, which is a contradiction because we removed the pair $(i+3,i-2)$.
Therefore $\vphi$ is a list homomorphism, which completes the proof.
 \end{proof}

Now we can prove \cref{lem:leaf-instance}.

\begin{proof}[Proof of \cref{lem:leaf-instance}]
Recall that the instance $(G,L)$ is reduced. 
By \cref{lem:lists-3}, it is enough to show that every vertex of $V_{\geq 3}$ has list of type $(2,2)$. 
First observe that for every vertex $u$ outside the cycle $C$, we have that $L(u)\subseteq \{i-2,i-1,i,i+1,i+2\}$ for some $i\in [2k]_0$. Indeed, by \cref{obs:cycle-dist}, there must be $i\in [2k]_0$ such that $\dist(u,c_i)=\dist(u,c_{i+1})=:\ell\leq \diam(G)=k+2$. 
Similarly as in the proof of \cref{thm:poly}, it holds that $\ell\in \{k+1,k+2\}$, so by \cref{obs:lists-dist-two}, $L(u)\subseteq \{i-2,i-1,i,i+1,i+2\}$, for some $i\in [2k]_0$.

Furthermore, observe that since for the branching rule (B2), if we could, we chose vertex $v$ whose distance from $C$ is at least $\lceil\frac{d}{2}\rceil$, each vertex of $G$ is at distance $\lfloor\frac{k+2}{2}\rfloor$ from $C$. Indeed, every vertex $u'$ that was in $N^{\leq k+2}_{G[V\geq 2]}[v]$ either is already precolored or has a precolored neighbor after guessing the coloring on $N^{\leq k+1}_{G[V\geq 2]}[v]$, and thus each such vertex $u'$ has list of size at most $2$. So for any vertex $u$ that is still in $V_{\geq 3}$, the shortest $u$-$v$ path (whose length is at most the diameter $\diam(G)$) should contain a vertex from $C$ and length of that path is at least $\dist(v,C)+\dist(u,C)$. So either all vertices outside $C$ were at distance at most $\lfloor\frac{k+2}{2}\rfloor$, or $v$ was at distance at least $\lceil\frac{k+2}{2}\rceil$, and thus 
$\dist(u,C)\leq \diam(G) - \dist(v,C) \leq k+2 - \lceil\frac{k+2}{2}\rceil=\lfloor\frac{k+2}{2}\rfloor$, so $u$ is at distance at most $\lfloor\frac{k+2}{2}\rfloor$ from $C$.

For $k=1$, we obtain that every vertex of $G$ is at distance at most $\lfloor\frac{3}{2}\rfloor=1$ from $C$. Therefore, by \cref{obs:lists-dist}, we obtain that every vertex of $G$ has list of size at most $2$, and thus $(G,L)$ can be solved in polynomial time by \cref{lem:red-2-sat}. So since now we can assume that $k\geq 2$.
If $k>2$, then $\lfloor\frac{k+2}{2}\rfloor<k$. Recall that by \cref{obs:lists-dist}, for a vertex $u$ that is at distance at most $k-1$ from a precolored vertex (and all vertices of $C$ are precolored), the set $L(u)$ is an independent set. Combining it with the fact that each list of a vertex in $V_{\geq 3}$ is contained in $\{i-2,i-1,i,i+1,i+2\}$, for some $i\in [2k]_0$, we obtain that for every $u\in V_{\geq 3}$ we have $L(u)\subseteq \{i-2,i,i+2\}$, for some $i\in [2k]_0$.
If $k=2$, then each vertex is at distance at most $2$ from $C$, and by \cref{obs:lists-dist}, we obtain that for every $u\in V_{\geq 3}$, the list $L(u)$ is of type $(2,2)$, i.e., $L(u)=\{i-2,i,i+2\}$ for some $i\in [2k]_0$.

Therefore, every vertex of $V_{\geq 3}$ has list of type $(2,2)$, and thus, by \cref{lem:lists-3}, the instance $(G,L)$ can be solved in polynomial time, which completes the proof.
\end{proof}

Now we are ready to prove \cref{thm:subexp}.

\begin{proof}[Proof of \cref{thm:subexp}]
Let $(G,L)$ be an instance of $\lhomo{C_{2k+1}}$ such that the diameter of $G$ is at most $k+2$. As in \cref{thm:poly}, first for every $i\in [2k]_0$, we check in polynomial time whether there is a list homomorphism $\vphi:(G,L)\to C_{2k+1}$ such that no vertex is mapped to $i$ -- recall that this can be done by \cref{lem:paths-poly}. If there is no such list homomorphism, we guess $2k+1$ vertices $c_0,\ldots,c_{2k}$ which will be colored so that $c_i$ is mapped to $i$. We add the edges $c_ic_{i+1}$ and we obtain an induced $(2k+1)$-cycle $C$ (if not, then we are dealing with a no-instance). Note that adding edges cannot increase the diameter and since the edges are added between vertices precolored with consecutive vertices, we obtain an equivalent instance.

Now for $(G,L)$ and $C$ as the fixed precolored $(2k+1)$-cycle we use the algorithm \texttt{Recursion Tree}, which by \cref{lem:recursion-tree} in time $2^{\Oh((n\log{n})^{\frac{k+1}{k+2}})}$ returns a tree $\cR$. Moreover, in order to solve the instance $(G,L)$ it is enough to solve every instance corresponding to a leaf of $\cR$ by \cref{lem:recursion-tree}, and
by \cref{lem:leaf-instance}, we can solve each such instance in polynomial time. Furthermore, since the size of $\cR$ is bounded by the running time, the instance $(G,L)$ can be solved in time $\exp\left({\Oh((n\log{n})^{\frac{k+1}{k+2}})}\right)\cdot n^{\Oh(1)}=\exp\left({\Oh((n\log{n})^{\frac{k+1}{k+2}})}\right)$, which completes the proof.
\end{proof}

We finish this section with a result which shows that it is sometimes possible to have a subexponential-time algorithm for $\lhomo{C_{2k+1}}$ for diameter-$(k+3)$ graphs, i.e., we show that $\lhomo{C_5}$ can be solved in polynomial time on diameter-$5$ graphs. It is not clear if this can be generalized to other values of $k$ -- here we use the fact that after standard branchings, every vertex is at distance at most $\lfloor\frac{k+3}{2}\rfloor$, which for $k=2$ is $2$, from a precolored vertex. By \cref{obs:lists-dist}, the list of every vertex is either of size at most two or of type $(2,2)$, and thus the instance can be solved in polynomial time by \cref{lem:lists-3}. For all larger values of $k$, we have $\lfloor\frac{k+3}{2}\rfloor\geq 3$, and thus the same argument does not work.

\begin{theorem}\label{thm:c5}
Every diameter-5 $n$-vertex instance $(G,L)$ of $\lhomo{C_5}$ can be solved in time $\Oh\left(\exp\left((n\log{n})^{\frac{4}{5}}\right)\right)$.
\end{theorem}

\begin{proof}
Let $(G,L)$ be an $n$-vertex instance of $\lhomo{C_{5}}$ such that $\diam(G)\leq 6$. First, as in \cref{thm:subexp}, for every $i\in \{0,1,2,3,4\}$, by applying \cref{lem:paths-poly}, we check in polynomial time whether there is a list homomorphism $\vphi:(G,L)\to C_5$ such that no vertex is mapped to $i$. If no, we guess $5$ vertices $c_0,c_1,c_2,c_3,c_4$ which will be colored so that $c_i$ is mapped to $i$. We add the edges $c_ic_{i+1}$ and we obtain an induced $5$-cycle $C$ (if not, then we are dealing with a no-instance). 

We use the algorithm \texttt{Recursion Tree} for $(G,L)$ and $C$, which by \cref{lem:recursion-tree} in time $2^{\Oh((n\log{n})^{\frac{4}{5}})}$ returns a tree $\cR$  such that in order to solve the instance $(G,L)$ it is enough to solve every instance corresponding to a leaf of $\cR$.

Let $(G',L')$ be an instance corresponding to a leaf in $\cR$. As in the proof of \cref{lem:leaf-instance}, since in (B2), if we could, we chose a vertex whose distance from $C$ is at least $\lceil \frac{5}{2}\rceil$, each vertex of $G'$ is at distance at most $\lfloor \frac{5}{2}\rfloor=2$ from $C$. By \cref{obs:lists-dist}, each vertex of $V_{\geq 3}$ has list of type $(2,2)$, and thus by \cref{lem:lists-3}, $(G',L')$ can be solved in polynomial time. This completes the proof.
\end{proof}

\section{Beyond odd cycles}
In this section we consider target graphs other than odd cycles. Instead, we focus on input graphs with diameter at most $2$. Since homomorphisms preserve edges, for graphs $G,H$, a homomorphism $\vphi: G\to H$ and a sequence of vertices $v_1,\ldots,v_k$ forming a path in $G$, the sequence $\vphi(v_1),\ldots,\vphi(v_k)$ forms a walk in $H$. Therefore, if $G$ has diameter at most $2$, we can assume that the target graph $H$ has also diameter at most $2$. The following observation is straightforward.

\begin{observation}\label{obs:diam-G-H}
    Let $G,H$ be graphs such that $G$ is connected. If there exists a homomorphism $\vphi: G \to H$, then the image $\vphi(V(G))$ induces in $H$ a subgraph with diameter at most $\diam(G)$.
\end{observation}

%By \cref{obs:diam-G-H}, in order to solve $\lhomo{H}$ on instances $(G,L)$ where $G$ is a diameter-$d$ graph, it is enough to guess the image, which must be a diameter-$d$ induced subgraph $H'$ of $H$, and determine whether at l there  yields the following corollary.

%\begin{corollary}\label{cor:diam-H}
%    Let $H$ be a graph and let $(G,L)$ be an instance of $\lhomo{H}$ such that $G$ is connected and $d=\diam(G)$. If for every induced subgraph $H'$ of $H$ such that $\diam(H')\leq d$, we can solve every diameter-$d$ instance of $\lhomo{H'}$ in polynomial time, then we can solve the instance $(G,L)$ in polynomial time.
%\end{corollary}

Now we prove \cref{thm:triangle-free}.

%\begin{theorem}\label{thm:lists-triangle-free}
%    Let $H$ be a simple triangle-free graph. Then $\lhomo{H}$ is polynomial-time solvable on diameter-$2$ graphs.
%\end{theorem}

\thmtriangle*

\begin{proof}
    Let $(G,L)$ be an instance of $\lhomo{H}$. 
    We guess the set of colors that will be used -- by \cref{obs:diam-G-H} they should induce a diameter-$2$ subgraph $H'$ of $H$. For each such $H'$, we guess $h'=|H'|$ vertices $v_1,\ldots,v_{h'}$ of $G$ that will be injectively mapped to $V(H')=\{x_1,\ldots,x_{h'}\}$.
For each tuple $(H',v_1,\ldots,v_{h'})$ such that $x_i\in L(v_i)$ for $i\in [h']$, we solve the instance $(G,L')$ of $\lhomo{H'}$, where $L'(v)=\{x_i\}$ for $v=v_i$, $i\in [h']$ and $L'(v)=L(v)$ otherwise. Note that $(G,L)$ is a yes-instance if and only if at least one instance $(G,L')$ is a yes-instance.

 First, for every edge $x_ix_j\in E(H')$, if $v_i,v_j$ are non-adjacent, we add the edge $v_iv_j$ to $G$ -- note that this operation is safe, since we cannot increase the diameter by adding edges and we only add edges between vertices that must be mapped to neighbors in $H'$. Therefore, we can assume that the set $V'=\{v_1,\ldots,v_{h'}\}$ induces a copy of $H'$ in $G$ (if not, then we have an extra edge, which means that we are dealing with a no-instance and we reject immediately). Furthermore, we exhaustively apply the reduction rules.

So from now on we assume that the instance $(G,L')$ is reduced. We claim that either $(G,L')$ is a no-instance or $V(G)=\{v_1,\ldots,v_{h'}\}$, i.e., after exhaustive application of the reduction rules, the graph $G$ is isomorphic to $H'$. Note that in the latter case we can return YES as an answer.

Suppose there is $v\in V(G)\setminus V'$.
Moreover, we choose such $v$ which is adjacent to some vertex of $V'$ (see \cref{fig:triangle-free}).
Suppose that there exists $\vphi: (G,L')\to H'$ and let $x_i=\vphi(v)$.
Then $v$ cannot be adjacent to $v_i$ since there are no loops in $H'$.
    Furthermore, the only neighbors of $v$ in $V'$ can be the neighbors of $v_i$.
    Suppose that there is $v_j\in N_G(v_i)\cap V'$ which is non-adjacent to $v$. 
    Since the diameter of $G$ is at most $2$, then there must be $u\in N_G(v)\cap N_G(v_j)$. 
    Observe that $u\notin V'$. 
    Indeed, $v$ does not have any neighbors in $V'\setminus N_G(v_i)$ and if $u\in N_G(v)$, then there is a triangle $uv_iv_j$ in a copy of $H'$, a contradiction. 
    Furthermore, it must hold that $\vphi(u)$ is adjacent to $x_i$ in $H'$ as $u$ is adjacent to $v$ and $\vphi(v)=x_i$, and similarly, $\vphi(u)$ must be adjacent to $x_j$ as $u$ is adjacent to $v_j$. 
    Then $\vphi(u)x_ix_j$ forms a triangle in $H'$, a contradiction. Thus $v$ must be adjacent to all vertices of $N(v_i)\cap V'$.

\begin{center}
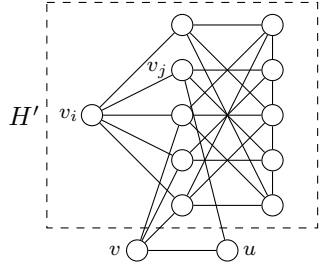
\begin{figure}[h]
    \centering
    \begin{tikzpicture}[every node/.style={draw,circle,fill=white,inner sep=0pt,minimum size=8pt},every loop/.style={},scale=0.6]
    \node[label=left:\footnotesize{$v_i$}] (vi) at (0,0) {}; 
    \node (v1) at (2,2) {};
    \node[label=left:\footnotesize{$v_j$}] (v2) at (2,1) {};
    \node (v3) at (2,0) {};
    \node (v4) at (2,-1) {};
    \node (v5) at (2,-2) {};
    \node (u1) at (4,2) {};
    \node (u2) at (4,1) {};
    \node (u3) at (4,0) {};
    \node (u4) at (4,-1) {};
    \node (u5) at (4,-2) {};
    \node[label=left:\footnotesize{$v$}] (v) at (1,-3) {};
    \node[label=right:\footnotesize{$u$}] (u) at (3,-3) {};
    \foreach \k in {1,2,3,4,5}
    {
    \draw (vi)--(v\k);
    }
    \foreach \k in {3,4,5}
    {
    \draw (v)--(v\k);
    }
    \draw (u)--(v2);
    \draw (u)--(v);

    \draw (v1)--(u1);
    \draw (v2)--(u2);
    \draw (u1)--(u2);
    \draw (v3)--(u3);
    \draw (u3)--(u4)--(u5)--(v5)--(u3);
    \draw (u2)--(u3);
    \draw (u5)--(v1)--(u3);
    \draw (u1)--(v3)--(u5);
    \draw (v4)--(u4);
    \draw (v2)--(u4);
    \draw (u2)--(v4);
    \draw (u1)--(v5);
    \draw[dashed] (-1,2.5)--(5,2.5)--(5,-2.5)--(-1,-2.5)--(-1,2.5);
    \node[draw=none, fill=none] (h) at (-1.5,0) {$H'$};

    \end{tikzpicture}
    \caption{\label{fig:triangle-free} The copy of $H'$ in $G$ and a vertex $v$ such that for some homomorphism $\vphi$, it holds $\vphi(v_i)=\vphi(v)$. We show that a vertex $u$ which is a common neighbor of $v$ and some neighbor $v_j$ of $v_i$ in the copy of $H'$ cannot exist.}
    \end{figure}
    \end{center}

Since \ref{red:edges} cannot be applied, each vertex of $L'(v)$ is adjacent to all vertices of $N_H(x_i)$.    
Moreover, since \ref{red:comp} cannot be applied, it holds that $L'(v)=\{x_i\}$.
Indeed, otherwise there is $x_{i'}\neq x_i$ such that $x_{i'}\in L'(v)$. Recall that $x_{i'}$ is adjacent to all vertices of $N_H(x_i)$. Therefore, $N_H(x_i)\subseteq N_H(x_{i'})$, and thus one of $x_{i},x_{i'}$ should have been removed from $L'(v)$ by \ref{red:comp}. 
Furthermore, since \ref{red:identify} cannot be applied, we must have $v=v_i\in V'$, a contradiction.
  This completes the proof.
\end{proof}

%\newpage
\section{Lower bound}
%\begin{theorem}
 %   Let $k\geq 2$. Then $\homo{C_{2k+1}}$ is NP-hard on diameter-$(2k+2)$ graphs (of radius $k+1$) and cannot be solved in subexponential time, unless the ETH fails.
%\end{theorem}

In this section we prove \cref{thm:hardness}.

\thmhardness*

\begin{proof}
   We reduce from $3$-\textsc{Sat}. Let $\phi$ be an instance with $n$ variables $x_1,\ldots,x_n$ and $m$ clauses $\gamma_1,\ldots, \gamma_m$. We construct $G$ as follows.

    We start with a $(2k+1)$-cycle $C$ on vertices $\{v_0,\ldots,v_{2k}\}$. For every clause $\gamma_j$ we add a copy of $C_{2k+1}$ on vertices $a_j,b_j,c^1_j,c^2_j,\ldots,c^{2k-1}_j$ and we add edges $v_1a_j, b_jv_2$. For every variable $x_i$, we add a copy of $C_{2k+1}$ with vertices $x_i^0,\ldots,x_i^{2k}$ and identify $x_i^0$ with $v_0$.

    For every clause $\gamma_j$, we fix ordering of its variables. Furthermore, for every variable $x_i$ of $\gamma_j$ we add a path $P_{ij}$ as follows. 
    \begin{enumerate}
        \item If $x_i$ is the first variable of $\gamma_j$ we add a path $P_{ij}$ on $2k+1$ vertices, identify its first vertex with $a_j$. Furthermore, for $\ell=2,\ldots,2k$, we make the $\ell$-th vertex of the path adjacent to $v_{\ell}$. Finally, if the occurence of $x_i$ in $C_j$ is positive, then we identify the $(2k+1)$th vertex of the path with $x_i^{2k}$. Otherwise, we identify the $(2k+1)$th vertex with $x_i^1$.
        \item If $x_i$ is the second variable of $\gamma_j$, then we add a path $P_{ij}$ on three vertices, identify the first vertex with $b_j$, and make the second vertex adjacent to $v_1$. Finally, if the occurence of $x_i$ in $C_j$ is positive we identify the third vertex of the path with $x_i^1$. Otherwise, we identify it with $x_i^{2k}$.
        \item If $x_i$ is the third variable of $\gamma_j$, then we add a path $P_{ij}$ on $k+2$ vertices, identify the first one with $c^{k}_j$, and make the $(k+1)$-th one adjacent to $v_1$. Finally, if the occurence of $x_i$ in $C_j$ is positive, then we identify the last vertex of the path with $x_i^{1}$, and otherwise, we identify it with $x_i^{2k}$.
    \end{enumerate}

    This completes the construction of $G$ (see \cref{fig:lower}). Note that $|V(G)|=\Oh(n+m)$. We first prove that $\phi$ is satisfiable if and only if $G\to H$.

    \begin{center}
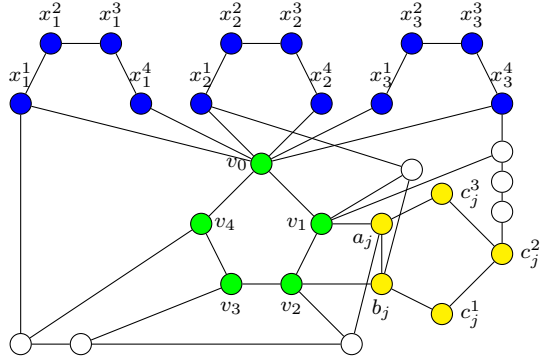
\begin{figure}[h]
    \centering
    \begin{tikzpicture}[every node/.style={draw,circle,fill=white,inner sep=0pt,minimum size=8pt},every loop/.style={},scale=0.8]
    \node[fill=green,label=left:\footnotesize{$v_0$}] (v0) at (0,0) {};
    \node[fill=green,label=left:\footnotesize{$v_1$}] (v1) at (1,-1) {};
    \node[fill=green,label=below:\footnotesize{$v_2$}] (v2) at (0.5,-2) {};
    \node[fill=green,label=below:\footnotesize{$v_3$}] (v3) at (-0.5,-2) {};
    \node[fill=green,label=right:\footnotesize{$v_4$}] (v4) at (-1,-1) {};
    \draw (v0)--(v1)--(v2)--(v3)--(v4)--(v0);
    \node[fill=yellow] (aj) at (2,-1) {};
    \node[draw=none,fill=none] (t) at (1.7,-1.3) {\footnotesize{$a_j$}};
    \node[fill=yellow,label=below:\footnotesize{$b_j$}] (bj) at (2,-2) {};
    \draw (v1)--(aj)--(bj)--(v2);
    \node[fill=yellow,label=right:\footnotesize{$c_j^1$}] (c1) at (3,-2.5) {};
    \node[fill=yellow,label=right:\footnotesize{$c_j^2$}] (c2) at (4,-1.5) {};
    \node[fill=yellow,label=right:\footnotesize{$c_j^3$}] (c3) at (3,-0.5) {};
    \draw (bj)--(c1)--(c2)--(c3)--(aj);
    %%%%%%%%%%%%%%%%%%%%%%%%%%%%%%%%%
    \node[fill=blue,label=above:\footnotesize{$x_2^1$}] (x11) at (-1,1) {};
    \node[fill=blue,label=above:\footnotesize{$x_2^2$}] (x12) at (-0.5,2) {};
    \node[fill=blue,label=above:\footnotesize{$x_2^3$}] (x13) at (0.5,2) {};
    \node[fill=blue,label=above:\footnotesize{$x_2^4$}] (x14) at (1,1) {};
    
    %%%%%%%%%%%%%%%%%%%%%%%%%%%%%
    \node[fill=blue,label=above:\footnotesize{$x_1^1$}] (x21) at (-4,1) {};
    \node[fill=blue,label=above:\footnotesize{$x_1^2$}] (x22) at (-3.5,2) {};
    \node[fill=blue,label=above:\footnotesize{$x_1^3$}] (x23) at (-2.5,2) {};
    \node[fill=blue,label=above:\footnotesize{$x_1^4$}] (x24) at (-2,1) {};
    %%%%%%%%%%%%%%%%%%%%%%%%%%%%%%%%%
    \node[fill=blue,label=above:\footnotesize{$x_3^1$}] (x31) at (2,1) {};
    \node[fill=blue,label=above:\footnotesize{$x_3^2$}] (x32) at (2.5,2) {};
    \node[fill=blue,label=above:\footnotesize{$x_3^3$}] (x33) at (3.5,2) {};
    \node[fill=blue,label=above:\footnotesize{$x_3^4$}] (x34) at (4,1) {};
    %%%%%%%%%%%%%%%%%%%%%%%%%%%%%%%%%%%
    \draw (v0)--(x11)--(x12)--(x13)--(x14)--(v0);
    \draw (v0)--(x21)--(x22)--(x23)--(x24)--(v0);
    \draw (v0)--(x31)--(x32)--(x33)--(x34)--(v0);

    \node (p11) at (1.5,-3) {};
    \node (p12) at (-3,-3) {};
    \node (p13) at (-4,-3) {};
    \draw (aj)--(p11)--(p12)--(p13)--(x21);
    \draw (p11)--(v2);
    \draw (p12)--(v3);
    \draw (p13)--(v4);
    %%%%%%%%%%%%%%%%%%%%%%%%%%
    \node (p21) at (2.5,-0.1) {};
    \draw (bj)--(p21)--(x11);
    \draw (p21)--(v1);
    %%%%%%%%%%%%%%%%%%%%%%%%%%%%%%%
    \node (p31) at (4,-0.8) {};
    \node (p32) at (4,-0.3) {};
    \node (p33) at (4,0.2) {};
    \draw (c2)--(p31)--(p32)--(p33)--(x34);
    \draw (p33)--(v1);

 \end{tikzpicture}
    \caption{\label{fig:lower} Construction of $G$ for $k=2$ and clause $\gamma_j=(\neg x_1\lor x_2 \lor \neg x_3)$. Green vertices belong to the cycle $C$, blue vertices are those introduced for variables, and yellow ones are those introduced for the clause $\gamma_j$. Remaining vertices belong to paths $P_{ij}$.}
    \end{figure}
    \end{center}
    
\setcounter{theorem}{3}
     \begin{claim}
       If $G\to C_{2k+1}$, then $\phi$ is satisfiable.
    \end{claim}

    \begin{claimproof}
        First suppose that there exists $\vphi: G\to C_{2k+1}$ and without loss of generality, assume that $\vphi(v_i)=i$ for $i\in [2k]_0$. Note that for every variable, there are exactly two ways of coloring its corresponding copy of $C_{2k+1}$, i.e., one with $\vphi(x_i^{2k})=2k$, and the other with $\vphi(x_i^{2k})=1$. We define a truth assignment $\psi$ of the variables, so that $x_i$ is true if and only if $\vphi(x_i^{2k})=2k$.

        Let us verify that $\psi$ is a satisfying assignment of $\phi$. Consider a clause $\gamma_j$ and its corresponding copy of $C_{2k+1}$. Observe that $a_j$ is adjacent to vertex precolored with $1$ and $b_j$ is adjacent to a vertex precolored with $2$. Therefore, the pair $(a_j,b_j)$ must be colored in one of three ways: $(0,1), (2,3), (2,1)$. First assume that $(a_j,b_j)$ is colored with $(0,1)$ and let $x_i$ be the first variable of $C_j$. 
        Recall that $P_{ij}$ is a path on $2k+1$ vertices with $a_j$ as the first vertex. Moreover, for every $\ell=2,\ldots,2k$, the $\ell$th vertex of $P_{ij}$ is adjacent to $v_{\ell}$ and we have $\vphi(v_{\ell})=\ell$. Thus $\ell$th vertex of $P_{ij}$ has to be mapped to one of $\{\ell-1,\ell+1\}$.
        Furthermore, $\vphi(a_j)=0$, so the second vertex of $P_{ij}$ has to be mapped to $1$. Then $\ell$th vertex has to be mapped to $\ell-1$, and thus the last vertex has to be mapped to $2k$.
        Recall that the last vertex of the path $P_{ij}$ is $x_i^{2k}$ if the occurence of $x_i$ is positive and $x_i^{1}$, otherwise. In both cases by the definition of the truth assignment $\psi$, $x_i$ satisfies $C_j$.

        Now assume that the pair $(a_j,b_j)$ is colored with $(2,3)$ and let $x_i$ be the second variable of $C_j$. Then observe that the consecutive vertices of $P_{ij}$ must be colored with $3,2,1$, respectively. Recall that the last vertex of $P_{ij}$ is $x_i^1$ if the occurence of $x_i$ is positive and $x_i^{2k}$, otherwise. In both cases by the definition of truth assignment $\psi$, $x_i$ satisfies $C_j$.

        Finally assume that the pair $(a_j,b_j)$ is colored with $(2,1)$ and let $x_i$ be the third variable of $C_j$. Observe that in this case $c_j^k$, which is also the first vertex of $P_{ij}$, must be colored with $k+2$. Since the $(k+1)$th vertex of $P_{ij}$ is adjacent to $v_1$ colored with $1$, it can be colored only with one of $0,2$. However, if it is colored with $0$, then there must be a walk in $C_{2k+1}$ from $0$ to $k+2$ of length exactly $k$, a contradiction. Therefore, the last vertex of $P_{ij}$ must be colored with $1$. As in the previous cases, $x_i$ satisfies $C_j$.
    \end{claimproof}

    \begin{claim}
        If $\phi$ is satisfiable, then $G\to C_{2k+1}$.
    \end{claim}

    \begin{claimproof}
        Let $\psi$ be a truth assignment satisfying $\phi$. We define $\vphi: G\to C_{2k+1}$ as follows. First, we set $\vphi(v_i):=i$. Moreover, for every variable $x_i$, we extend $\vphi$ to the vertices of the cycle introduced for $x_i$, so that $\vphi(x_i^1)=1$ if $\psi(x_i)=1$, and $\vphi(x_i^1)=2k$ otherwise. Furthermore, for each clause $\gamma_j$ we fix one variable $x_i$ that satisfies $\gamma_j$ in assignment $\psi$. Then, if the first variable satisfies $\gamma_j$, we color $(a_j,b_j)$ with $(0,1)$, if it is the second variable, we color $(a_j,b_j)$ with $(2,3)$, and if it is the third one, we color $(a_j,b_j)$ with $(2,1)$. We extend $\vphi$ to the remaining vertices of the cycle introduced for $\gamma_j$ in the only possible way. 

        Observe that so far, $\vphi$ respect all the edges. It remains to extend $\vphi$ to vertices of paths $P_{ij}$.
        Consider such a path $P_{ij}$ introduced for a clause $\gamma_j$ and its variable $x_i$.

        \textbf{Case 1: $x_i$ is the first variable of $\gamma_j$.} In this case $P_{ij}$ is a path on $2k+1$ vertices with $a_j$ as the first vertex and $x_i^1$ or $x_i^{2k}$ as the last vertex depending on the sign of the occurence of $x_i$ in $\gamma_j$. Moreover, for $\ell=2,\ldots,2k$, $\ell$th vertex of $P_{ij}$ is adjacent to $v_\ell$, so it has to be mapped to one of $\ell-1,\ell+1$. If $a_j$ is colored with $2$, then $\vphi$ can be extended to $P_{ij}$ so that the $2k$ consecutive vertices are mapped respectively to $2,3,4\ldots,2k,0$. Now note that this respects all the edges of $P_{ij}$ as the last vertex is colored either with $1$ or $2k$, both adjacent to $0$. So now assume that $a_j$ is colored with $0$. By the definition of $\vphi$, the variable $x_i$ must satisfy $\gamma_j$, and thus the last vertex of $P_{ij}$ has to be mapped to $2k$. Then we can extend $\vphi$ to $P_{ij}$ so that the consecutive vertices of $P_{ij}$ are mapped respectively to $0,1,\ldots,2k$.

        \textbf{Case 2: $x_i$ is the second variable of $\gamma_j$.} Then $P_{ij}$ is a path on $3$ vertices with $b_j$ being the first vertex. Moreover, the second vertex is adjacent to $v_1$, so it has to be mapped to $0$ or $2$. If $b_j$ is mapped to $1$, then we set $\vphi$ on the second vertex of $P_{ij}$ to $0$, which is adjacent to both $1,2k$, and thus $\vphi$ respects the edges of $P_{ij}$. If $b_j$ is mapped to $3$, then the variable $x_i$ satisfies $\gamma_j$ and thus the last vertex of $P_{ij}$ has to be mapped to $1$. Therefore we can set $\vphi$ on the second vertex of $P_{ij}$ to $2$.

        \textbf{Case 3: $x_i$ is the third variable of $\gamma_j$.} In this case $P_{ij}$ is a path on $k+2$ vertices with $c_j^k$ being the first vertex and $(k+1)$th vertex adjacent to $v_1$. Note that if $(a_j,b_j)$ is mapped to $(0,1)$, $(2,3)$, or $(2,1)$ , then $c_j^k$ is mapped respectively to $k+1$, $k+3$, $k+2$. In the first case, we can extend $\vphi$ so that consecutive $k+1$ vertices of $P_{ij}$ are mapped respectively to $k+1,k+2,\ldots,2k,0$ and since the $k+1$th vertex is mapped to $0$ adjacent to both $1,2k$, $\vphi$ respects all the edges of $P_{ij}.$ In the second case we extend $\vphi$ to $P_{ij}$ so that consecutive $k+1$ vertices of $P_{ij}$ are mapped respectively to $k+3,k+4,\ldots,2k,0,1,0$ and again $\vphi$ respects the edges of $P_{ij}$. So let us consider the third case. Recall that this happens when $x_i$ satisfies $C_j$ and thus the last vertex of $P_{ij}$ is mapped to $1$. We extend $\vphi$ so that the consecutive vertices of $P_{ij}$ are mapped respectively to $k+2,k+1,k,\ldots,2,1$, which clearly respects the edges of $P_{ij}$. This completes the proof of the claim.
    \end{claimproof}

Now we show that the radius (and thus the diameter) of $G$ is bounded.

\begin{claim}
The radius of $G$ is at most $k+1$.
\end{claim}

\begin{claimproof}
We show that each vertex is at distance at most $k+1$ from $v_1$.
It holds for every vertex of the cycle $C$ and every vertex that is adjacent to $C$. The remaining vertices are those introduced for variables, the vertices of cycles introduced for clauses, and those of paths $P_{ij}$ introduced for clauses and their third variables.

The vertices introduced for variables are at distance at most $k$ from $v_0$, and thus at most $k+1$ from $v_1$.
For a cycle introduced for a clause $\gamma_j$, each its vertex is at distance at most $k$ from $a_j$, which is adjacent to $v_1$.
Finally, it remains to check the internal vertices of the paths introduced in the third case. Recall that each such a path consist of $k+2$ vertices ($k$ internal vertices) and the $(k+1)$-th vertex is adjacent to $v_1$. This completes the proof of the claim.
    \end{claimproof}

Therefore the $\homo{C_{2k+1}}$ problem is \NP-hard on diameter-$(2k+2)$ graphs. Moreover, since $|V(G)|=\Oh(n+m)$, there is no algorithm solving $\homo{C_{2k+1}}$ in time $2^{o(|V(G)|)}\cdot |V(G)|^{\Oh(1)}$, unless the ETH fails. This completes the proof.
\end{proof}

\section{Conclusion}
In this paper we studied the computational complexity of $\homo{C_{2k+1}}$ problem on bounded-diameter graphs. We proved that for $k\geq 2$, the $\homo{C_{2k+1}}$ problem can be solved in polynomial-time on diameter-$(k+1)$ graphs and we gave subexponential-time algorithm for diameter-$(k+2)$ graphs. Furthermore, we showed that $\homo{C_{2k+1}}$ cannot be solved in subexponential time on diameter-$(2k+2)$ graphs, unless the ETH fails. We also proved that $\homo{H}$ for triangle-free graph $H$, can be solved in polynomial time on diameter-$2$ graphs.

The main open problem in this area remains the question of whether $3$-\textsc{Coloring} on diameter-$2$ graphs can be solved in polynomial time. However, as more reachable, we propose the following future research directions.
(i) Is the $\homo{C_{2k+1}}$ problem \NP-hard for the subexponential-time cases, i.e., diameter-$(k+2)$ graphs? 
(ii) What is the computational complexity of $\homo{C_{2k+1}}$ on diameter-$d$ graphs for $d\in \{k+3,\ldots,2k-1\}$?

\smallskip

\textbf{Acknowledgements.} We are grateful to Paweł Rzążewski for inspiring and fruitful discussions.

\newpage
\bibliographystyle{plain}
\bibliography{main}

\end{document}